\documentclass[10pt]{iopart}
\usepackage{iopams} 
\usepackage{url}
\usepackage{graphicx}
 \expandafter\let\csname equation*\endcsname\relax
 \expandafter\let\csname endequation*\endcsname\relax 
\usepackage{amsmath}
\usepackage{amsfonts}
\usepackage{amsopn}

\usepackage{amsmath,amsthm,amsfonts,amssymb,bm}
\usepackage{amsbsy}
\usepackage{multirow}
\usepackage{todonotes}
\usepackage[colorlinks,allcolors=blue]{hyperref}
\usepackage[per-mode = fraction]{siunitx}

\usepackage{mathrsfs}
\usepackage{listings}
\usepackage{graphicx}
\usepackage{mathtools}
\usepackage[caption=false]{subfig}
\usepackage{grffile}
\usepackage{float}
\usepackage{empheq}
\usepackage{tabu}
\usepackage{xcolor}

\usepackage{color}

\usepackage{algorithm}
\usepackage{algorithmic}

\theoremstyle{plain}
\newtheorem{theorem}{Theorem}[section]

\newtheorem{cor}[theorem]{Corollary}
\newtheorem{as}{Assumption}
\newtheorem{definition}[theorem]{Definition}

\theoremstyle{remark}
\newtheorem*{remark}{Remark}

\usepackage{float}

\usepackage{cite}

\newcommand{\prior} {\pi_{\mbox{\tiny prior}}}

\newcommand{\post}{\pi_{\mbox{\tiny post}}}

\renewcommand{\vec}[1]{{\mathchoice
                     {\mbox{\boldmath$\displaystyle{#1}$}}
                     {\mbox{\boldmath$\textstyle{#1}$}}
                     {\mbox{\boldmath$\scriptstyle{#1}$}}
                     {\mbox{\boldmath$\scriptscriptstyle{#1}$}}}}

\renewcommand{\epsilon}{\varepsilon}
\renewcommand{\phi}{\varphi}
\newcommand{\map}{\textbf{MAP}}

\DeclareMathOperator{\jac}{\mathbf{J}}

\DeclareMathOperator{\argmin}{argmin}
\DeclareMathOperator{\argmax}{argmax}

\newcommand{\bbr}{\boldsymbol{r}}
\newcommand{\norm}[1]{\left\lVert#1\right\rVert}

\newcommand{\bra}[1]{\langle #1 \rangle}

\newcommand{\curBK}[1]{ {\left\{ #1 \right\}} }

\newcommand{\bR}{\mathbb{R}}

\newcommand{\calL}{\mathcal{L}}

\newcommand{\calD}{\mathcal{D}}

\newcommand{\vG}{\vec{G}}
\newcommand{\vN}{\vec{N}}

\newcommand{\vx}{\vec{x}}
\newcommand{\vy}{\vec{y}}
\newcommand{\vz}{\vec{z}}

\newcommand{\normal}{\mathcal{N}}

\usepackage{acro}
\acsetup{format/long=\it, format/short=\sc}
\DeclareAcronym{PAT}{short=PAT, long=photoacoustic tomography}
\DeclareAcronym{NFR}{short=NFR, long=Normalizing Flow Regularization}
\DeclareAcronym{TV}{short=TV, long=Total Variation}
\DeclareAcronym{CT}{short=CT, long=Computer Tomography}
\DeclareAcronym{CNN}{short=CNN, long=Convolutional Neural Network}
\DeclareAcronym{MRI}{short=MRI, long=Magnetic Resonance Imaging}
\DeclareAcronym{PSNR}{short=PSNR, long=Peak Signal-to-Noise Ratio}
\DeclareAcronym{SSIM}{short=SSIM, long=Structural Similarity Index}
\DeclareAcronym{RRA}{short=RRA, long=relative reconstruction accuracy}
\DeclareAcronym{GAN}{short=GAN, long=Generative Adversarial Network}
\DeclareAcronym{NF}{short=NF, long=Normalizing Flow}
\DeclareAcronym{MAP}{short=MAP, long=maximum a posteriori}
\DeclareAcronym{MLE}{short=MLE, long=maximum likelihood estimation}
\DeclareAcronym{UBP}{short=UBP, long=universal back-projection}
\usepackage{xspace}
\newcommand{\Unet}{\texttt{U-Net}\xspace}
\newcommand{\Glow}{\texttt{Glow}\xspace}

\begin{document}

\title[NF regularization for PAT]{{Normalizing flow regularization for photoacoustic tomography}}
\author{Chao Wang$^{\dagger}$, Alexandre H. Thiery$^{\dagger}$}
\address{$^{\dagger}$ Department of Statistics and Data Science, National University of Singapore}
\ead{chaowyww@gmail.com, a.h.thiery@nus.edu.sg}

\vspace{10pt}
\begin{indented}
\item[]Aug 2022
\end{indented}

\begin{abstract}
\label{abs}
Proper regularization is crucial in inverse problems to achieve high-quality reconstruction, even with an ill-conditioned measurement system. This is particularly true for three-dimensional photoacoustic tomography, which is computationally demanding and requires rapid scanning, often leading to incomplete measurements. Deep neural networks, known for their efficiency in handling big data, are anticipated to be adept at extracting underlying information from images sharing certain characteristics, such as specific types of natural or medical images.
We introduce a \ac{NFR} method designed to reconstruct images from incomplete and noisy measurements. The method involves training a normalizing flow network to understand the statistical distribution of sample images by mapping them to Gaussian distributions. This well-trained network then acts as a regularization tool within a Bayesian inversion framework. Additionally, we explore the concept of adaptive regularization selection, providing theoretical proof of its admissibility.
A significant challenge in three-dimensional image training is the extensive memory and computation requirements. We address this by training the normalizing flow model using only small-size images and applying a patch-based model for reconstructing larger images. Our approach is model-independent, allowing the reuse of a well-trained network as regularization for various imaging systems. Moreover, as a data-driven prior, \ac{NFR} effectively leverages the available dataset information, outperforming artificial priors. This advantage is demonstrated through numerical simulations of three-dimensional photoacoustic tomography under various conditions of sparsity, noise levels, and limited-view scenarios.
\end{abstract}

\section{Introduction}
\label{sec:1}

Photoacoustic tomography as a nonionizing imaging modality is promising for clinical applications~\cite{Nie2014,Valluru2016,Beard2011,Wang2009a,Xia2014}. 
A short laser pulse illuminates a biological tissue, creating a small amount of heating throughout the tissue. This process leads to a thermoelastic expansion that generates a pressure wave that propagates through the object. Crucially, the initial pressure field created by the laser depends on the local properties of the biological tissue. Consequently, reconstructing the initial pressure field can be useful in obtaining diagnostic information of clinical interest.
The reconstruction is obtained through a set of transducers surrounding the biological tissue that can measure the time-dependent pressure field. The inverse problem that consists of reconstructing the initial pressure field from these measurements is the topic of this text.
For fast data acquisition, which is crucial in clinical settings, only incomplete and noisy measurements can be collected. Regularization methods are consequently of paramount importance for addressing the associated ill-conditioned inverse problem. 
To exploit the statistical properties that scan of biological data are expected to satisfy (e.g. smooth background, sharp edges, the existence of low-rank or sparse structures), standard regularization techniques such as the use of $\ell_p$ norms, \ac{TV} regularization, or nuclear norms have been widely adopted to improve the quality of reconstruction. However, these general and analytically tractable regularization methods are unlikely to efficiently exploit the structures that characterize the scans in specific subfields (e.g., brain fMRI scans have different statistical properties than chest X-ray scans). Using recent advances in generative modeling parametrized with deep neural networks, this article proposes a data-driven approach to construct appropriate regularization methods leveraging large data sets of complete measurements.

Deep learning approaches are already playing a crucial role in image reconstruction~\cite{Chen2017,Wang2016a,Ongie2020,Aggarwal2019,McCann2017}. This includes post-processing-based methods, model-motivated models, and deep learning-based regularization. Post-processing-based methods are models that can transform a rough reconstruction into a more accurate one: once the post-processing model is fitted to data, which can be a computationally intensive procedure, deploying the model is generally computationally cheap and straightforward to implement. However, this class of methods generally relies on the existence of a large {\it training} dataset of pairs of input data and the corresponding accurate reconstruction. In medical settings, constructing such datasets can be difficult due to the associated costs and clinical constraints. Furthermore, post-processing methods can generally only be deployed for data obtained with the same imaging system as the one used for creating the training dataset. 
Model-motivated methods are a class of algorithms that learn (part of) an iterative optimization process for reconstruction purposes: this can lead to important computational savings while maintaining the reconstruction accuracy~\cite{Chang2017,Dittmer2020,Baguer2020,Banert2021,Hammernik2018,Adler2018,Adler2017b,Zhang2020}.
Finally, deep learning-based regularization approaches consist of using training data to build appropriate regularization terms~\cite{Aggarwal2019, Romano2017, duff2021}; the methods proposed in this text belong to this class of methods.

In photoacoustic tomography, deep learning methods already produce remarkable results~\cite{Hauptmann2020a,Grohl2021,Kim2020,Schwab2018}. For example, Convolutional Neural Networks \ac{CNN} have been used for designing post-processing methods to improve the accuracy of rough reconstructions obtained with standard methods such as \ac{UBP} \cite{PhysRevEXu} or adjoint-based approaches~\cite{Guan2021,Guan2020a,Antholzer2019a,Mohammad2018,Grohl2018,Guan2020,Kim2020,Lan2020}.
For model-motivated methods, due to the computationally demanding nature of estimating the forward and backward processes, each iteration is typically trained separately~\cite{Hauptmann2018}.
In \cite{Li2020,Haltmeier2019,Romano2017}, deep neural networks are utilized as regularizers to remove artifacts, and a complete convergence analysis of the proposed algorithms is also established.
These recently developed methodologies demonstrate the potential of data-driven approaches in the image reconstruction of photoacoustic tomography \ac{PAT}. It is important to note that, unlike in \ac{CT} or \ac{MRI} where large datasets are widely available, obtaining large datasets of paired images for emerging imaging technologies like photoacoustic tomography remains challenging. Therefore, for photoacoustic tomography, it is crucial to design methods that (1) do not rely on datasets of paired images and (2) are independent of the measurement system and noise statistics.

In this article, we propose the use of a normalizing flow~\cite{Kingma2018} as a regularizer to address the ill-posedness in image reconstruction with limited-view measurements. A trained normalizing flow is designed to map a distribution of images (i.e., training data) to a tractable \textit{base distribution}, typically chosen as an isotropic Gaussian. This procedure offers a data-driven approach to building a prior distribution, which can subsequently be used within a Bayesian inversion framework~\cite{Papamakarios2021,Siahkoohi2021}. Compared to more standard algorithms that typically rely on a small number of hand-crafted features and generic regularization terms, this data-driven approach more effectively leverages the statistical properties of the distribution of images being reconstructed. Additionally, the method is agnostic to the imaging system and noise statistics; only a dataset of \textit{ground truth} images is necessary. Once trained, the model can be applied within different measurement settings.

In this paper, we describe the \ac{NFR} method, a Bayesian inversion algorithm for three-dimensional photoacoustic tomography. Here, a normalizing flow is used to construct a prior distribution of images using a training dataset of clear images. Our contributions are as follows:

\begin{itemize}
\item We show that normalizing flow successfully learns the distribution of the training data and performs effectively as a regularizer within iterative methods. \ac{NFR} method achieves competitive reconstruction compared to state-of-the-art methods such as \ac{TV} regularization iterative methods and \ac{CNN} post-processing method. 
\item We investigate the tuning of the regularization parameter. We provide a rule for adaptively selecting the regularization parameter and prove its admissibility.
\item To mitigate the computational challenges associated with processing large three-dimensional images, we develop a \textit{patch-based} approach~\cite{Knyaz2019} that involves randomly selecting several patches across the three-dimensional image.
\end{itemize}

The rest of this paper is organized as follows. In Section \ref{sec:PAT}, we describe the Bayesian framework for solving the inverse problem of \ac{PAT}. Some background material on \ac{PAT} is contained in Section \ref{sec:PAT.IP}, and normalizing flow models are introduced in Section \ref{sec:NF}. The well-posedness of \ac{NFR} is analyzed in Section \ref{sec:glow} and \ref{sec:3.2} using \Glow as an example. A strategy to select the regularization parameter is proposed and studied in Section \ref{sec:3.3}. The corresponding algorithm is described in Section \ref{sec:4}: the adaptive algorithm to select the regularization parameter is described and analyzed in Section \ref{sec:4.1}, and the patch-based approach is discussed in Section \ref{sec:4.2}. Section \ref{sec:5} compares in different sparsity and noise levels the \ac{NFR} with standard \ac{TV} methods and a \ac{CNN} post-processing methods. Section \ref{sec:6} concludes the article.

\section{Bayesian inverse problems in \ac{PAT}}
\label{sec:PAT}

\subsection{Inverse problem in \ac{PAT}}
\label{sec:PAT.IP}

As explained in the introduction, in three-dimensional \ac{PAT}, the illumination of a biological tissue $\Omega \subset \mathbb{R}^3$ with a laser creates a spatially inhomogeneous pressure field that depends on the local biological properties of the tissue. The initial pressure field $p_0(\bbr) \equiv p(\bbr, t=0)$, for spatial coordinate $\bbr \in \bR^3$, then propagates through the tissues before being measured by a set of $M \geq 1$  transducers surrounding the tissue. In other words, the transducers measure the time-dependent pressure field $p(\bbr, t)$ for $t \in [0,T]$ at a finite set of spatial locations $\{\bbr_i\}_{i=1}^M$ located on the boundary of $\Omega$.
Assuming that the biological tissue is acoustically homogeneous (i.e., the acoustic velocity $c_0$ is constant throughout the medium) and acoustically linear and non-attenuating on the spatial and time-relevant scales, as is approximately the case in most clinically relevant situations, the evolution of the pressure field $p(\bbr, t)$ is described by a constant coefficient wave equation,
\begin{equation}
  \label{eq:wave}
  \left\{
  \begin{aligned}
    &\partial^2_{tt} \, p(\bbr, t) = c_0^2 \, \Delta p(\bbr,t)&& \qquad (\bbr,t)\in \mathbb{R}\times (0,T],\\
    &p(\bbr, 0) = p_0(\bbr), && \qquad \bbr\in \Omega,\\
    &p(\bbr,0) = 0, && \qquad\bbr\in \mathbb{R}^3/\Omega,\\
    &\frac{\partial p}{\partial t}(r,0) = 0, && \qquad \bbr \in \Omega.
  \end{aligned}
  \right.
\end{equation}
This initial value problem for the wave equation is well studied, and the pressure field at time $t \in [0, T]$ is given by Kirchhoff's formula. We have that
\begin{equation}
 p(\bbr, t) =  \partial_t \, \Big\{ c_0 \, t \times ( \mathcal{A} p_0)(\bbr, c_0 \, t)  \Big\},
\end{equation}
where, for any $\bbr \in \bR^3$ and radius $R > 0$, the quantity $( \mathcal{A} p_0)(\bbr, R)$ is the average of the initial pressure field $p_0$ on the sphere centered at $\bbr$ with radius $R$,
\begin{equation}
 ( \mathcal{A} p_0)(\bbr, R) \equiv  \int_{\vec{v} \in \mathbb{S}_2} p_0(\bbr + R \, \vec{v}) \, \sigma(d \vec{v}),
\end{equation}
with $\sigma(d \vec{v})$ the uniform probability distribution on the unit sphere $\mathbb{S}_2 \subset \bR^3$. The averaging operator $\mathcal{A}$is related to the Radon-transform that has been extensively studied in computational tomography. After discretization, this linear inverse problem can be formulated as
\begin{equation}
\label{eq:IP}
\vy = \vec{F} \vx + \vec{\eta}
\end{equation}
for a vector $\vx \in \bR^{D_{\vx}}$ representing the discretization of the unknown initial acoustic pressure field $p_0(\bbr)$ and a vector $\vy \in \bR^{D_{\vy}}$ that represents the collected data contaminated by additive noise $\vec{\eta} \in \bR^{D_{\vy}}$. In the remainder of this text, we assume that the additive error term $\vec{\eta}$ follows a centered Gaussian distribution with covariance matrix $\Gamma_{\eta}$ although the our proposed methodology readily generalizes to other noise distributions. Note that the matrix $\vec{F}$ that represents the discretization of the forward operator \eqref{eq:wave} is typically extremely ill-conditioned and is never actually computed or stored in memory.

Combining a prior distribution with density $\prior(\vx)$ (with respect to the Lebesgue measure in $\bR^{D_{\vx}}$) that encodes the probabilistic assumptions on $\vx \in \bR^{D_{\vx}}$ with the probabilistic model described in Equation \eqref{eq:IP} leads to the posterior distribution $\post(d\vx)$ whose density (with respect to the Lebesgue measure in $\bR^{D_{\vx}}$) reads
\begin{align} \label{eq.post}
\post(\vx) = \tfrac{1}{Z} \, \prior(\vx) \, \exp \curBK{-\frac{1}{2} \|\vy - \vec{F} \vx\|^2_{\Gamma^{-1}_{\eta}}}
\end{align}
for a normalization constant $Z = \int \, \prior(\vx) \, \exp \curBK{-\frac{1}{2} \|\vy - \vec{F} \vx\|^2_{\Gamma^{-1}_{\eta}}} \, d \vx$. We have used the standard weighted norm notation $\|\vx\|^2_{\Gamma^{-1}_{\eta}} \equiv \bra{\vx, \Gamma^{-1}_{\eta}\vx}$. In this text, we are primarily interested in leveraging data to design an appropriate prior distribution $\prior$ that encodes as faithfully as possible the statistical properties of the discretized initial pressure field $\vx \in \bR^{D_{\vx}}$ and compute the \ac{MAP} estimate $\vx_{\map} = \argmax \post(\vx)$. Indeed, the \ac{MAP} estimate $\vx_{\map} \in \bR^{D_{\vx}}$ can also be expressed as the minimizer of the objective function,
\begin{align}
\vx \; \mapsto \; \frac{1}{2} \, \|\vy - \vec{F} \vx\|^2_{\Gamma^{-1}_{\eta}}  + \vec{R}(\vx)
\end{align}
where the negative log-prior density $\vec{R}(\vx) \equiv -\log \prior(\vx)$ takes the role of a regularization term. The negative log-likelihood $\frac{1}{2} \, \|\vy - \vec{F} \vx\|^2_{\Gamma^{-1}_{\eta}}$ is a data-misfit term.
Furthermore, we only consider the case where the covariance of the additive noise $\eta$ is isotropic with variance $\lambda > 0$, i.e. we have $\Gamma_{\eta} \equiv \lambda \, \mathbf{I}_{D_{\vy}}$, although more complex covariance structures would be straightforward to accommodate with the methods described in this text. The \ac{MAP} estimate $\vx_{\map}$ can consequently be expressed as the minimizer of the objective
\begin{align}
\label{eq.MAP}
\calL(\vx; \lambda) \; \equiv \; \frac{1}{2} \, \|\vy - \vec{F} \vx\|^2  + \lambda \, \vec{R}(\vx) .
\end{align}
The automated tuning of the parameter $\lambda > 0$ is important in practice and adaptive strategies are presented and theoretically analyzed in Sections \ref{sec:3.3} and \ref{sec:4.1}.

\subsection{Normalizing flow-based image prior}
\label{sec:NF}

In order to design a sensible prior distribution that translates the statistical assumptions on the discretization $\vx \in \bR^{D_{\vx}}$ of the initial pressure field $p_0$, a natural approach consists in fitting a statistical model to a large dataset of ground truth samples $\calD \equiv \{ \vx_i \}_{i=1}^{|\calD|}$. In this text, we illustrate the proposed method on \ac{CT} data. For a family of probability densities $\pi_{\theta}$ on $\bR^{D_{\vx}}$ parametrized by a parameter $\theta \in \Theta \subset \bR^{D_{\Theta}}$, this program can be carried out by \ac{MLE}, i.e. minimization of the negative log-likelihood function
\begin{align*}
\theta \mapsto  \sum_{\vx_i \in \calD} -\log \pi_{\theta}(\vx_i).
\end{align*}
Indeed, density estimation in high-dimensional scenarios faces significant challenges. Moreover, designing flexible parametrized densities in $\mathbb{R}^{D_{\vx}}$ to accurately model the statistics of biological tissue scans is not straightforward. Here, we leverage \ac{NF} parametrized with 3D \ac{CNN} since this type of neural architectures have the correct inductive biases for efficiently modeling spatially extended samples with local correlation structures. A \ac{NF} model on $\bR^{D_{\vx}}$ defines a family of densities $\pi_{\theta}(\vx)$ by leveraging a family of diffeomorphisms, i.e. continuously differentiable bijective mappings, and a base probability density $\pi_{B}(\vz)$ on $\bR^{D_{\vx}}$. More precisely, for a (generative) diffeomorphism $\vG_{\theta}: \bR^{D_{\vx}} \to \bR^{D_{\vx}}$ and its inverse $\vN_{\theta}: \bR^{D_{\vx}} \to \bR^{D_{\vx}}$, i.e.
\begin{align*}
\vG_{\theta} \circ \vN_{\theta}(\vx) = \vx \qquad \text{for all} \quad (\vx, \theta) \in \bR^{D_{\vx}} \times \Theta,
\end{align*}
the distribution $\pi_{\theta}$ is defined as the push-forward of the base distribution $\pi_{B}$ through the generative mapping $\vG_{\theta}$. For any (bounded) measurable function $\phi: \bR^{D_{\vx}} \to \bR$ we have that
\begin{align*}
\int \phi(\vx) \, \pi_{\theta}(\vx) \, d\vx = \int \phi [ \vG_{\theta}(\vz) ]\, \pi_{B}(\vz) \, d\vz.
\end{align*}
A sample $\vx \sim \pi_\theta$ can be constructed by first sampling from the base distribution $\vz \sim \pi_B$ and then set
\begin{align*}
\vx \; = \; \vG_{\theta}(\vz).
\end{align*}
The mapping $\vN_{\theta}$ is {\it normalizing} in the sense that it maps the distribution $\pi_\theta$ to the base distribution $\pi_{B}$. The density $\pi_{\theta}$ reads
\begin{align}
\label{eq.pi.theta}
\pi_{\theta}(\vx) \; = \; \pi_B[\vN_{\theta}(\vx)] \times \left| \det \jac_{\vN_{\theta}}(\vx) \right|
\end{align}
where $\jac_{\vN_{\theta}}(\vx)$ denotes the Jacobian matrix of $\vN_{\theta}$ at $\vx$. In practice, the normalizing mapping $\vN_\theta$ is obtained by composing $K \geq 1$ simpler diffeomorphism mappings,
\begin{align*}
\vN_{\theta} \; \equiv \; \vN_{\theta, K} \circ \cdots \circ  \vN_{\theta, 2} \circ \vN_{\theta, 1}.
\end{align*}
In this text and as is standard in applications, the base distribution $\pi_{B}$ is chosen as the standard centered Gaussian distribution $\normal( \mathbf{0}, \mathbf{I})$ in $\bR^{D_{\vx}}$ so that $-\log \pi_{B}(\vz) = \frac12 \, \| \vz \|^2 + \textrm{(constant)}$. Importantly, the mappings $\{ \vN_{\theta, k}  \}_{k=1}^{K}$ are constructed so that the Jacobian determinants $\det(\jac_{\vN_{\theta,k}})$ are straightforward to compute. This implies that
\begin{align}
\label{eq.chain.rule}
-\log \pi_\theta(\vx) = \frac12 \| \vN_{\theta}(\vx) \|^2 - \sum_{k=1}^K \log \left| \det[\jac_{\vN_{\theta,k}}(\vx_k)] \right| + \textrm{(constant)}
\end{align}
with $\vx_1 = \vx$ and $\vx_{k+1} = \vN_{k}(\vx_k)$ for $1 \leq k \leq K-1$. Section \ref{sec:glow} describes more precisely the neural parametrization of the transformations $\vN_{\theta, k}$. The parameter $\theta \in \Theta$ that describes the neural weights parametrizing the normalizing function $\vN$ is obtained by maximum likelihood estimation,
\begin{align}
\label{eq.fitting.flow}
\theta_\star \; = \; \argmin_{\theta \in \Theta} \; \curBK{ \; \theta \mapsto \sum_{\vx_i \in \mathcal{D}}   \Big( \frac{1}{2}\|   \vN_{\theta}(\vx_i) \|^2 -  \log \left| \det[\jac_{\vN_{\theta}}(\vx_{i})] \right| \Big) \;}.
\end{align}
Once the neural parameter $\theta_\star \in \Theta$ is obtained, this defines a data-driven regularization functional,
\begin{align*}
\vec{R}(\vx) \equiv   \frac{1}{2} \| \vN_{\theta_\star}(\vx)\|^2 -  \log \left| \det[\jac_{\vN_{\theta_\star}}(\vx)] \right|
=
-\log \pi_{\theta_\star}(\vx) + \textrm{(constant)}
\end{align*}
that approximates, up to an irrelevant additive constant, the negative log-density of the ground truth samples $\calD \equiv \{ \vx_i \}_{i=1}^{|\calD|}$.
The \ac{MAP} estimate $\vx_{\map}$ of the discretized initial pressure field $\vx \in \bR^{D_{\vx}}$ can consequently be obtained by minimizing the loss
\begin{align}
\label{eq:objective}
\calL(\vx; \lambda) \; \equiv \; \frac{1}{2} \, \|\vy - \vec{F} \vx\|^2  + \lambda \,  \Big(  \frac{1}{2} \| \vN_{\theta_\star}(\vx) \|^2 -  \log \left| \det[\jac_{\vN_{\theta_\star}}(\vx)] \right| \Big).
\end{align}
This minimization can be carried out with standard adaptive stochastic-gradient-descent methods.

\section{Analysis of normalizing flow-based regularization (NFR)}
\label{sec:3}

In this section, the well-posedness (existence, stability, and convergence) of \ac{NFR} is discussed. For concreteness, the \Glow neural architecture \cite{Kingma2018} is used in Section \ref{sec:glow} and \ref{sec:3.2}, although our proposed method is agnostic to the choice of neural architecture. Importantly, an adaptive tuning strategy to select the regularization parameter is presented and analyzed in Section \ref{sec:3.3}.

\subsection{\Glow model}
\label{sec:glow}

Given dataset $\calD \equiv \{ \vx_i \}_{i=1}^{|\calD|}$, the normalizing flow \Glow model can be used to for density estimation; the parameters $\theta_\star$ of the \Glow model are obtained through Maximum Likelihood Estimation as described in Equation \eqref{eq.fitting.flow}.
Empirical evidences show that \Glow can lead to significant improvements when compared to other standard flow models such as \texttt{NICE} \cite{Hatt2009} or \texttt{RealNVP} \cite{Dinh2017}. A \Glow model $\vN_\theta$ with parameter $\theta$ consists of multiple steps of flow $f_{\theta,k}\ (k=1, 2, \dots, K)$ that are connected by squeeze layers mappings $q_{\theta,k}\ (k=1, 2, \dots, K)$ and split layers $h_{\theta,k}\ (k=1, 2, \dots, K)$. The model can be summarized as follows,
\begin{equation}
    \label{eq:NF-consist-of}
    \begin{aligned}
    & \vN_{\theta} \; \equiv \; \vN_{\theta, K} \circ \cdots \circ  \vN_{\theta, 2} \circ \vN_{\theta, 1},\\
    &\vN_{\theta,k} \equiv h_{\theta,k} \circ f_{\theta,k} \circ q_{\theta,k}\  (k = 1, 2, \dots, K),\\
    \end{aligned}
\end{equation}
where $h_{\theta, K}=\vec{I}$.
In \eqref{eq:NF-consist-of}, $f^1_{\theta,k}$ is an activation normalization layer, $f^2_{\theta,k}$ is an invertible $1\times 1$ convolution and $f^3_{\theta,k}$ is an affine coupling layer. We have
\begin{subequations}
    \label{eq:step-flow}
\begin{empheq}{align}
    & f_{\theta,k} = f^3_{\theta,k} \odot f^2_{\theta,k} \odot f^1_{\theta,k}, \label{eq:step-flow-a}\\
    & \vx^{k,1} = f^1_{\theta,k}(\vx^{k,0}) = \vec{s}_{\theta,k}\odot \vx^{k,0} + \vec{b}_{\theta,k},\label{eq:step-flowa-b}\\
    & \vx^{k,2} = f^2_{\theta,k}(\vx^{k,1}) = \bm{W}_{\theta,k}\vx^{k,1},\label{eq:step-flow-c}\\
    & \vx^{k,3} = f^3_{\theta,k}(\vx^{k,2}) = 
    \begin{pmatrix}
      g^1_{\theta,k}\odot \vx_a^{k,2} + g^2_{\theta,k}\\
      \vx^{k,2}_b  
    \end{pmatrix},\label{eq:step-flow-d}
\end{empheq}
\end{subequations}
where the inputs of $f_{\theta,k}^1$, $f_{\theta,k}^2$ and $f_{\theta,k}^3$ are denoted by $\vx^{k,0}$, $\vx^{k,1}$ and $\vx^{k,2}$ with $\vx^{k,0}=q_{\theta,k}(\vx_k)$ and $h_{\theta,k}(\vx^{k,3})=\vx_{k+1}$. Note that $\bm{W}_{\theta,k}$ in \eqref{eq:step-flow-c} is an invertible matrix product. Affine coupling layer \eqref{eq:step-flow-d} acts on the two halves $\vx^{k,2}_a$ and $\vx^{k,2}_b$ that are splits along the channel of $\vx^{k,2}$. The quantities $g^1_{\theta,k}$ and $g^2_{\theta,k}$ are defined as
\begin{align*}
[\log g^1_{\theta,k}, g^2_{\theta,k}] \; = \;  \texttt{NN}(\vx^{k,2}_b),
\end{align*}
where $\log g^2_{\theta,k}$ and $g^1_{\theta,k}$ are the two parts of the output of the network \texttt{NN}, separated by channels. The neural network \texttt{NN} consists of three convolutional layers and nonlinear activation functions \texttt{RELU}. The reader is referred to \cite{Kingma2018} for details.

\subsection{Well-posedness of the regularization}
\label{sec:3.2}
Once the parameters $\theta_\star$ of the \Glow model are obtained, these parameters are fixed. The reconstruction of the initial pressure field is then carried out by minimizing the following regularized objective,
\begin{equation}
\label{eq:obj-glow}
\left\{
\begin{aligned}
\calL(\vx; \lambda, \vy) &\; \equiv \; \frac{1}{2} \, \|\vy - \vec{F} \vx\|^2_2  +\lambda\, \vec{R}(\vx)\\
\vec{R}(\vx) &\; \equiv \; \frac{1}{2} \| \vN_{\theta_\star}(\vx)\|^2 -  \log \left| \det[\jac_{\vN_{\theta_\star}}(\vx)] \right|,
\end{aligned}
\right.
\end{equation}
where parameter $\lambda > 0$ quantifies the amount of regularization
To proceed with the analysis of this regularized objective, any solution $\vx^+ \in \bR^{D_x}$ that satisfies
\begin{equation}
\label{eq:deep-prior-solu}
\vec{R}(\vx^+) = \inf \curBK{ \vec{R}(\vx) \; : \; \vec{F}\vx=\vy }.
\end{equation}
is referred to as a {\it deep-prior solution}. In other words, among all the solutions to the equation $\vec{F}\vx=\vy$, deep-prior solutions are the ones that minimize the regularization functional $\vec{R}$.

\begin{theorem}
\label{the:1}
Consider the regularized objective \eqref{eq:obj-glow}. The following properties hold:
\begin{enumerate}
\item {\bf Existence:} for any $\vy\in \bR^{D_{\vy}}$ and $\lambda>0$, there exists at least one minimizer of the regularized objective $\vx \mapsto \calL(\vx; \lambda, \vy)$;
\item {\bf Stability:} for any $\lambda > 0$ and sequence $\{(\vx_k, \vy_k)\}_{k \geq 0}$ such that $\vy_k \rightarrow \vy \in \bR^D_y$ and $\vx_k \in \argmin \calL(\vx; \lambda, \vy_k)$, there exists a subsequence of $\{ \vx_k \}_{k \geq 0}$ that converges to a minimizer of the regularized objective $\vx \mapsto \calL(\vx; \lambda, \vy)$.
\item {\bf Convergence:} for $\vx^*\in X$ and $\vy^*=\vec{F}\vx^*$, let a sequence $\vy_k$ satisfies $\norm{\vy^*-\vy_k}_2\leq \delta_k\rightarrow 0(k\rightarrow \infty)$, and regularization parameters satisfy $\lambda_k\rightarrow 0$ and $\delta_k^2/\lambda_k\rightarrow 0$, then the sequence of the minimizers of $\calL(\vx; \lambda_k, \vy_k)$ will at least has a subsequence that converges to a minimizer of $\calL(\vx;\lambda=0,\vy)$. The limits of convergent subsequences are deep-prior solutions of $\vec{F}\vx =\vy$ if the solutions of $\vec{F}\vx=\vy$ exist.
\end{enumerate}
\end{theorem}

\begin{proof}
The forward operator $\vec{F}$ is a finite-dimensional linear operator from $\bR^{D_\vx}$ to $\bR^{D_\vy}$. Furthermore, the regularization functional $\vec{R}(\vx) = \frac{1}{2} \| \vN_{\theta_\star}(\vx)\|^2 -  \log \left| \det[\jac_{\vN_{\theta_\star}}(\vx)] \right|$ is continuous with $\lim_{\|\vx\| \to \infty} \vec{R}(x) = \infty$.
Consequently, the existence, stability and convergence properties of the regularized functional \eqref{eq:obj-glow} follow from \cite{scherzer2009variational} (Chapter 5) and \cite{Li2020}(Theorem 2.6).

\end{proof}

\begin{remark}
A similar study of regularization with deep neural artchitecture can be found in~\cite{Li2020}. Weak convergence is discussed for sequentially lower semi-continuous regularizers $\vec{R}$ defined on infinite-dimensional space and conditions for strong convergence to hold are discussed.
\end{remark}

\subsection{Parameter selection strategy}
\label{sec:3.3}

This section describes an adaptation strategy for setting the regularization parameter $\lambda > 0$.
Its admissibility is proved in Theorem \ref{the:2} while its existence is proved in Theorem \ref{the:exist-rx-lambda}.
For any sample $\vx^*\in \bR^{D_\vx}$, it is assumed that the noise is $\vec{\vec{\eta}}\in \bR^{D_\vy}$, and the observation is $\vy_{\vec{\vec{\eta}}} = \vec{F} \vx^* + \vec{\vec{\eta}}$. According to Theorem \ref{the:1}, a minimizer of $\calL(\vx; \lambda, \vy_\vec{\eta})$ exists and it is denoted as
\begin{equation}
\label{eq:3.4}
\vx_\lambda^\vec{\vec{\eta}} \in 
\argmin \curBK{ \vx \; \mapsto \; \frac{1}{2}\norm{\vec{F}\vx-\vy_\vec{\vec{\eta}}}_2^2+\lambda \, \vec{R}(\vx)}.
\end{equation}
Since the noise level is unknown in applications of practical interest, the regularization parameter $\lambda$ can be difficult to choose. We propose to set $\lambda=\lambda(\vec{\eta})$ chosen so that
\begin{equation}
\label{eq:choice-lambda}
\vec{R}(\vx_{\lambda}^\vec{\eta})=\vec{R}(\vx^*).
\end{equation}
This strategy is referred to as {\it{regularizer consistency}} in the rest of this article.

\begin{definition}
\label{def:admissible}
For the inversion from $\vy_{\vec{\eta}} := \vec{F}\vx + \vec{\eta}$ to $\vx$ by the minimization \eqref{eq:3.4}, a regularization strategy $\lambda = \lambda(\vec{\eta})$ is called \textit{admissible} if when $\norm{\vec{\eta}}_2\rightarrow 0$, any minimizer $\vx_{\lambda}^{\vec{\eta}}$ satisfies 
\begin{equation}
\label{eq:admissible}
   \norm{\mathcal{P}_{\textrm{null}(\vec{F})^\bot}(\vx_{\lambda}^{\vec{\eta}}-\vx)}_2 \rightarrow 0, 
\end{equation}
where $\textrm{null}(\vec{F})^\bot$ is the orthogonal complement of the nullspace of $\vec{F}$, and $\mathcal{P}$ is the projection operator. If $\vec{F}$ is injective, the projection operator in \eqref{eq:admissible} can be removed.       
\end{definition}

The following theorem establishes that the adaptive rule described by Equation \eqref{eq:choice-lambda} is admissible.

\begin{theorem}
\label{the:2}
Assuming $\vx\in \mathcal{D}\subset \mathbb{R}^{D_\vx}$ and $\vy = \vec{F}\vx^*$, if for any $\vec{\vec{\eta}}$ with $\norm{\vec{\eta}}_2< \norm{\vy}_2$, there exists $\lambda$ such that there exist $\vx_\lambda^{\vec{\eta}}\in \argmin_\vx \calL(\vx;\lambda,\vy_\vec{\eta})$ satisfying \eqref{eq:choice-lambda}, then $\eqref{eq:choice-lambda}$ is an admissible rule. Furthermore, if $\vx^*$ is a deep-prior solution of $\vec{F}\vx^*$, when $\vec{\eta} = \vec{\eta}_k\rightarrow 0$ and $\lambda_k = \lambda(\vec{\eta}_k)$ such that \eqref{eq:choice-lambda} holds, there exist accumulation points of minimizers $\vx_{\lambda_k}^{\vec{\eta}_k}$. These accumulation points are deep-prior solutions of $\vec{F}\vx = \vy$.
\end{theorem}
\begin{proof}
Let $\vec{\vec{\eta}}_k\rightarrow 0$, $\vy_k = \vy+\vec{\vec{\eta}}_k$. The regularization that satisfies \eqref{eq:choice-lambda} is denoted by $\lambda_k$ and the minimizers of $\calL(x;\lambda_k,\vy_k)$ are denoted by $\vx_k$. According to the definition of minimizers of $\calL$, 
\begin{equation}
    \label{eq:minimizer}
    \frac{1}{2}\norm{\vec{F}\vx_k-\vy_k}_2^2 + \lambda_k \, \vec{R}(\vx_k)\leq \frac{1}{2}\norm{\vec{F}\vx^*-\vy_k}_2^2 + \lambda_k \, \vec{R}(\vx^*).
\end{equation}
Because $\vx_k$ satisfies $\vec{R}(\vx_k)=\vec{R}(\vx^*)$, then 
\begin{equation}
    \label{eq:proof-the}  
    \norm{\vec{F}\vx_k-\vy_k}_2^2\leq \norm{\vec{\eta}_k}_2^2\rightarrow 0.
\end{equation}
Then it follows that
\begin{align*}
\norm{\vec{F}\mathcal{P}_{\textrm{null}(\vec{F})^\bot}(\vx_k-\vx^*)}_2 \leq \norm{\vec{F}(\vx_k - \vx^*)}_2 \leq \norm{\vec{F}\vx_k -\vy_k}_2 + \norm{\vec{\eta}_k}_2 \rightarrow 0.
\end{align*}
Since $\vec{F}$ is injective when as a mapping from $\mathcal{P}_{\textrm{null}(\vec{F})^\bot}$ to $\mathbb{R}^{D_\vy}$, we can conclude that
\begin{align*}
\norm{\mathcal{P}_{\textrm{null}(\vec{F})^\bot}(\vx_k - \vx^*)}_2 \rightarrow 0.
\end{align*}
Furthermore, when $\vx^*$ is a deep-prior solution of $\vec{F}\vx=\vy$. Due to that $\calL(\vx_k;\lambda_k,\vy_k)$ is bounded and $\vec{R}(\vx)$ is continuous, the level set $M_t(\lambda, \vy):=\{\vx:\calL(\vx;\lambda, \vy)\leq t\}$ is sequentially pre-compact for any $t>0$, $\lambda>0$ and $\vy \in \bR^{D_\vy}$, which induces that $\vx_k$ has accumulation points. Assume $\vx'$ is an accumulation point, then there is a subsequence $\vx_{k(n)}$ satisfying $\vx_{k(n)}\rightarrow \vx'\in \bR^{D_{\vx}}$. Then $\vec{F}\vx'=\vy$. From the continuity of $\vec{R}$, 
\begin{equation}
    \label{eq:rx}
    \vec{R}(\vx')= \lim_{n\rightarrow \infty} \vec{R}(\vx_{k(n)}) = \vec{R}(\vx^*).
\end{equation}
Therefore, $\vx'$ is a deep-prior solution of $\vec{F}\vx=y$.
\end{proof}

Then, the existence of $\lambda$ that satisfies \eqref{eq:choice-lambda} will be discussed. First of all, the monotonicity of $\vec{R}$ with respect to $\lambda$ is studied.

\begin{theorem}
\label{the:monotonicity-of-R}
Let $\vx_\lambda=\vx(\lambda)$ be a minimizer of $\calL(\vx;\lambda, \vy)$ for regularization parameter $\lambda>0$. Then the following holds:
\begin{enumerate}
    \item[(1)] $\vec{R}(\vx_\lambda)$ is monotonically decreasing. $\norm{\vec{F}\vx_{\lambda}-\vy}_2^2$ and $\calL(\vx_\lambda;\lambda,\vy)$ are monotonically increasing with respect to $\lambda$.
    \item[(2)] There exists $\overline{\vx}$ such that $\lim_{\lambda\rightarrow \infty}\vec{R}(\vx_\lambda)= \vec{R}(\overline{\vx})=\min_{\vx \in \bR^{D_\vx}} \vec{R}(\vx)$.
    \item[(3)] There exists $\underline{\vx}$ such that $\lim_{\lambda\rightarrow 0}\vec{R}(\vx_\lambda)=\vec{R}(\underline{\vx})$.
\end{enumerate}
\end{theorem}

\begin{proof}
\begin{enumerate}
    \item[(1)] For two regularization parameters $\lambda_1<\lambda_2$, the corresponding minimizers are $\vx_{\lambda_1}$ and $\vx_{\lambda_2}$, it follows that 
\begin{equation}
  \label{eq:3.6}
  \begin{aligned}
    & \frac{1}{2}\norm{\vec{F}\vx_{\lambda_1}-\vy}_2^2 + \lambda_1 \, \vec{R}(\vx_{\lambda_1})  \leq  \frac{1}{2}\norm{\vec{F}\vx_{\lambda_2}-\vy}_2^2  + \lambda_1 \, \vec{R}(\vx_{\lambda_2}) \\
    & \frac{1}{2}\norm{\vec{F}\vx_{\lambda_2}-\vy}_2^2  + \lambda_2 \, \vec{R}(\vx_{\lambda_2}) \leq  \frac{1}{2}\norm{\vec{F}\vx_{\lambda_1}-\vy}_2^2  + \lambda_2 \, \vec{R}(\vx_{\lambda_1}).
  \end{aligned}
\end{equation}
Adding two inequalities in \eqref{eq:3.6}, we can obtain
\begin{equation}
  \label{eq:3.7}
 (\lambda_1-\lambda_2)(\vec{R}(\vx_{\lambda_1}) - \vec{R}(\vx_{\lambda_2}))\leq 0,
\end{equation}
which concludes that $\vec{R}(\vx_{\lambda_1})\geq\vec{R}(\vx_{\lambda_2})$. Moreover, from the first inequality of \eqref{eq:3.6}, it is obtained that $\norm{\vec{F}\vx_{\lambda_1}-\vy}_2^2\leq \norm{\vec{F}\vx_{\lambda_2}-\vy}_2^2$. Due to $\vec{R}(\vx_{\lambda_2})>0$, according to the first inequality of \eqref{eq:3.6}, it is obtained that
\begin{align*}
\calL(\vx_{\lambda_1};\lambda_1,\vy)
& \leq  \frac{1}{2}\norm{\vec{F}\vx_{\lambda_2}-\vy}_2^2  + \lambda_1 \, \vec{R}(\vx_{\lambda_2}) \\
& \leq \frac{1}{2}\norm{\vec{F}\vx_{\lambda_2}-\vy}_2^2  + \lambda_2 \, \vec{R}(\vx_{\lambda_2})
= \calL(\vx_{\lambda_2};\lambda_2,\vy).
\end{align*} 

\item[(2)] Because $\vec{R}$ is coercive and positive, that is $\norm{\vx}_2\rightarrow \infty$ implies $\vec{R}(\vx)\rightarrow \infty$, there exists $c$ such that $c:=\text{inf }_{\vx\in \ell^2(\bR^{D_\vx})}\{\vec{R}(\vx)\}$. Then, from the continuity of $\vec{R}$, there exists $\overline{\vx}$ such that $\vec{R}(\overline{\vx})=\inf_{\vx\in \ell^2(\bR^{D_\vx})}\inf \vec{R}(\vx)$. For any $\lambda>0$, from the definition of the minimizer, it follows that 
\begin{equation}
    \label{eq:mono-x-lambda}
    0\leq\frac{1}{2}\norm{\vec{F}\vx_\lambda-\vy}_2^2+\lambda\vec{R}(\vx_\lambda)\leq\frac{1}{2}\norm{\vec{F}\overline{\vx}-\vy}_2^2+\lambda\vec{R}(\overline{\vx}).
\end{equation}
Dividing \eqref{eq:mono-x-lambda} by $\lambda$, it is obtained that
\begin{align*}
\vec{R}(\vx_\lambda)\leq\frac{1}{2\lambda}\norm{\vec{F}\overline{\vx}-\vy}_2^2+\vec{R}(\overline{\vx}).
\end{align*}
When $\lambda\rightarrow \infty$, it is obtained that
\begin{align*}
\liminf_{\lambda\rightarrow \infty}\vec{R}(\vx_\lambda)\leq\limsup_{\lambda\rightarrow \infty}\vec{R}(\vx_\lambda)\leq\lim_{\lambda\rightarrow \infty}\frac{1}{2\lambda}\norm{\vec{F}\overline{\vx}-\vy}_2^2+\vec{R}(\overline{\vx})=\vec{R}(\overline{\vx}).
\end{align*}
From the definition of $\overline{\vx}$, $\vec{R}(\overline{\vx})\leq \liminf_{\lambda\rightarrow \infty}\vec{R}(\vx_\lambda)$. Therefore, 
\begin{align*}
\lim_{\lambda\rightarrow \infty}\vec{R}(\vx_\lambda) = \vec{R}(\overline{\vx}).
\end{align*}

\item[(3)] Let $\vx'$ be the one with minimal $\vec{R}(\vx)$ that minimizes $\norm{\vec{F}\vx-\vy_\vec{\eta}}_2$, that is $\vec{R}(\vx')=\inf \{\vec{R}(\vx): \norm{\vec{F}\vx-\vy}_2=\min_\vx \norm{\vec{F}\vx-\vy}_2\}$. Then from the definition of $\calL(\vx;\lambda,\vy)$, it follows that
\begin{align*}
\frac{1}{2}\norm{\vec{F}\vx_\lambda-\vy}_2^2+\lambda \vec{R}(\vx_\lambda)\leq \frac{1}{2}\norm{\vec{F}\vx'-\vy}_2^2+\lambda \vec{R}(\vx').
\end{align*}
Because $\norm{\vec{F}\vx_\lambda-\vy}_2^2\geq \norm{\vec{F}\vx'-\vy}_2^2$, $\vec{R}(\vx_\lambda)\leq \vec{R}(\vx')$. Then when $\lambda\rightarrow 0$, it follows that
\begin{align*}
\lim_{\lambda\rightarrow 0}\norm{\vec{F}\vx_\lambda-\vy}_2^2\leq \norm{\vec{F}\vx'-\vy}_2^2 + \lim_{\lambda\rightarrow 0}2\lambda (\vec{R}(\vx')-\vec{R}(\vx_\lambda)) = \norm{\vec{F}\vx'-\vy}_2^2.
\end{align*}
Therefore, $\lim_{\lambda\rightarrow 0}\norm{\vec{F}\vx_\lambda-\vy}_2^2 = \norm{\vec{F}\vx'-\vy}_2^2$. For the decreasing sequence $\lambda_k\rightarrow 0\ (k\rightarrow \infty)$, $\norm{\vec{F}\vx_{\lambda_k}-\vy}_2^2$ is decreasing and $\vec{R}(\vx_{\lambda_k})$ is increasing. And it follows that  $\vec{R}(\vx_{\lambda_k}) \leq \vec{R}(\vx'):=M< \infty$. From the continuity of $\vec{R}$, the level set $\{\vx:\vec{R}(\vx)\leq M\}$ is pre-compact, then there exists a convergent subsequence $\vx_{\lambda_{k'}}$. Let $\underline{\vx}=\lim_{k'\rightarrow \infty}\vx_{\lambda_{k'}}$. It follows that
\begin{align*}
\vec{R}(\underline{\vx})= \lim_{k'\rightarrow \infty} \vec{R}(\vx_{\lambda_{k'}})\leq  \vec{R}(\vx').
\end{align*}
From the definition of $\vx'$, $\vec{R}(\underline{\vx})\geq\vec{R}(\vx')$. Therefore, $\vec{R}(\underline{\vx})=\vec{R}(\vx')$. By $\vec{R}(\vx_{\lambda_k})$ is increasing, it is obtained that $\lim_{k\rightarrow \infty} \vec{R}(\vx_{\lambda_k})= \lim_{k'\rightarrow \infty}\vec{R}(\vx_{\lambda_{k'}})=\vec{R}(\underline{\vx})=\vec{R}(\vx')$. Therefore, it is concluded that $\lim_{\lambda\rightarrow 0}\vec{R}(\vx_\lambda)=\vec{R}(\underline{\vx})$.

\end{enumerate}

\end{proof}

Theorem \ref{the:monotonicity-of-R} prove the boundedness of $\vec{R}(\vx_{\lambda})$ for the minimizer of $\calL(\vx;\lambda, \cdot)$. To further study the range and the relationship of $\vec{R}(\vx^*)$ and $\vec{R}(\vx_\lambda^\vec{\eta})$, we make the following assumption of the target image $\vx^*$, noise $\vec{\eta}$ and regularization $\vec{R}$.

\begin{as}
\label{as:noise-assumption}
Noise $\vec{\eta}\in \bR^{D_\vy}$ and the target image $\vx^*\in \ell^\infty(\bR^{D_\vx})$ satisfy:
\begin{enumerate}
    \item[(1)] If $\vec{\eta}$ is nonzero, $\vec{\eta}\notin \text{Ran }(\vec{F})^\bot$.
    \item[(2)] $min_{\vx\in \bR^{D_\vx}}\vec{R}(\vx)>0$.
    \item[(3)] For target image $\vx^*$ and noise $\vec{\eta}$ that satisfy (1) and (2), if $ \vx_\vec{\eta}$ is a minimizer of $\norm{\vec{F}\vx-\vy_{\vec{\eta}}}_2^2$, $\vec{R}(\vx_{\vec{\eta}})\geq \vec{R}(\vx^*)$ holds.
\end{enumerate}
\end{as}

Actually, (1) of Assumption \ref{as:noise-assumption} implies that for non-zero $\vec{\eta}$ there exists nonzero $\vx$ such that $\vec{F}\vx = \mathcal{P}_{\text{Ran}(\vec{F})}\vec{\eta}$, where $\mathcal{P}_{\text{Ran}(\vec{F})}$ denotes the projection to the range of $\vec{F}$, $\text{Ran}(\vec{F})$. In fact, from the properties of the solution of wave equation \eqref{eq:wave}, there exists $T_0>0$ such that $p(\vec{r},t)=0$ for $0<t<T_0$ and $p(\vec{r}, t)\neq 0$ for $t> T_0$, which means that the ultrasound signal starts to be detected at $T_0$. If $\vec{\eta}\neq \vec{0}$ and $\vec{\eta}\in \text{Ran}(\vec{F})^\bot$, then all nonzero elements of $\vec{\eta}$ corresponds to $t<T_0$, which is not worth studying. Normally, noise $\vec{\eta}$ is assumed isotropic Normal distribution. Therefore, (1) of Assumption \ref{as:noise-assumption} is reasonable.

As for (2) of Assumption \ref{as:noise-assumption}, from the definition of $\vec{R}$, $\vec{R}(\vx)\geq 0$ holds for any $\vx\in \bR^{D_\vx}$. Due to the randomness of $\vx$, $\pi_{\theta_\star}(\vx)<1$, then $\vec{R}(\vx)>0$ holds.

Under (1) and (2) of Assumption \ref{as:noise-assumption}, (3) of Assumption \ref{as:noise-assumption} implies that the distribution of the noisy images reconstructed from the noisy observation without regularization, that is the minimizer of $\norm{\vec{F}\vx-\vy_{\vec{\eta}}}_2^2$, is far away from the distribution of the target images. Otherwise, it means that noisy and noise-free data are indistinguishable, which is not our concern.

\begin{cor}
\label{cor:exist-rx}
Let $\vx_{\lambda}^\vec{\eta}$ denote minimizers of $\calL(\vx;\lambda,\vy_\vec{\eta})$, in which $\vy_\vec{\eta}=\vec{F}\vx^*+\vec{\eta}$ is the noisy observation. If Assumption \ref{as:noise-assumption} holds, there exists $\underline{\lambda}\geq\overline{\lambda}$ such that there exist $\vx_{\underline{\lambda}}^\vec{\eta}\in \argmin_\vx\calL(\vx;\underline{\lambda},\vy)$ and $\vx_{\overline{\lambda}}^\vec{\eta}\in \argmin_\vx\calL(\vx;\overline{\lambda},\vy)$ satisfy
$\vec{R}(\vx_{\underline{\lambda}}^\vec{\eta})\leq\vec{R}(\vx^*)$ and $\vec{R}(\vx_{\overline{\lambda}}^\vec{\eta})\geq\vec{R}(\vx^*)$.
\end{cor}

\begin{proof}
From Theorem \ref{the:monotonicity-of-R}, there exists $\underline{\vx}$ such that 
\begin{align*}
\lim_{\lambda\rightarrow \infty}\vec{R}(\vx_{\lambda}^\vec{\eta})=\vec{R}(\underline{\vx})=\min_{\vx \in \bR^{D_\vx}} \vec{R}(\vx).
\end{align*}
So for any $\epsilon>0$, there is $\underline{\lambda}$ such that when $\lambda>\underline{\lambda}$, $\vec{R}(\vx_\lambda^\vec{\eta}) < \vec{R}(\underline{\vx})+\epsilon$. Then $\vec{R}(\vx_\lambda^{\vec{\eta}}) < \vec{R}(\underline{\vx})+\epsilon  \leq \vec{R}(\vx^*) + \epsilon$. Therefore, $\vec{R}(\vx_{\lambda}^{\vec{\eta}}) \leq \vec{R}(\vx^*)$. Similarly, there exists $\overline{\vx}$ such that 
\begin{align*}
\lim_{\lambda\rightarrow 0}\vec{R}(\vx_\lambda^\vec{\eta})=\vec{R}(\overline{\vx}).
\end{align*}
From the proof of Theorem \ref{the:monotonicity-of-R}, $\overline{\vx}$ minimizes $\norm{\vec{F}\vx^*-\vy_\vec{\eta}}_2$ and $\vec{R}(\overline{\vx})=\inf\{\vec{R}(\vx):\norm{\vec{F}\vx-\vy_\vec{\eta}}_2=\min_x\norm{\vec{F}\vx-\vy_\vec{\eta}}\}$. From the (3) of Assumption \ref{as:noise-assumption}, it follows that $\vec{R}(\overline{\vx})\geq\vec{R}(\vec{\vx^*})$. Similarly, there exists $\overline{\lambda}$ such that $\vec{R}(\vx_{\overline{\lambda}}^\vec{\eta})\geq \vec{R}(\vx^*)$.
\end{proof}

In practice, only the bounded minimizers are useful. Next, the minimizers in a bounded region $\Omega=[0, M]^{D_\vx}\subset \bR^{D_\vx}$ are studied. Since $\vec{R}$ is coercive, there exists $M$ such that $\overline{\vx}$ and $\underline{\vx}$ in Theorem \ref{the:monotonicity-of-R} can be found in $\Omega$.  

\begin{theorem}
\label{the:continuous-lambda}
For fixed $\lambda>0$, for any $\epsilon>0$, there exists $\delta>0$ such that when $|\overline{\lambda}-\lambda|<\delta$, for any $\vx_{\lambda}\in \argmin_{\vx\in \Omega} \calL(\vx;\lambda,\vy)$, there exists $\vx_{\overline{\lambda}}\in \argmin_{\vx\in \Omega} \calL(\vx;\overline{\lambda},\vy)$ that satisfy $\norm{\vx_\lambda-\vx_{\overline{\lambda}}}_2<\epsilon$.
\end{theorem}
\begin{proof}
First, due to the uniform continuity of $\vec{F}$ and $\vec{R}$ on the closed region $\Omega$, there exists constants $C_1, C_2, C_3>0$ that do not depend on $\vx$ such that 
\begin{align*}
\begin{aligned}
& \left|\norm{\vec{F}\vx_1-\vy}_2^2/2-\norm{\vec{F}\vx_2-\vy}_2^2/2\right|<C_1 \norm{\vx_1-\vx_2}_2\\
& \left|\vec{R}(\vx_1)-\vec{R}(\vx_2)\right|<C_2 \norm{\vx_1-\vx_2}_2,\\
& \vec{R}(\vx)<C_3.
\end{aligned}
\end{align*}
Then for any $\vx_1, \vx_2\in \Omega$ and $\lambda_1$ and $\lambda_2$, $\calL(\vx;\lambda,\vy)$ satisfy
\begin{equation}
    \label{eq:L-continuous}
    \begin{aligned}
    & \left| \calL(\vx_1;\lambda) - \calL(\vx_2;\lambda)\right|\leq (C_1 + C_2 \lambda) \norm{\vx_1-\vx_2}_2\\
    & \left| \calL(\vx;\lambda_1)-\calL(\vx;\lambda_2)\right| \leq C_3 |\lambda_1-\lambda_2|.
    \end{aligned}
\end{equation}
For any $\lambda>0$, there are three cases of the number of the minimizers of $\calL(\vx;\lambda,\vy)$ in $\Omega$: unique, finite, or infinite. These cases are discussed respectively in the following.
\begin{enumerate}
    \item[(1)]When the minimizer $\vx_{\lambda}=\argmin_{\vx\in \Omega}\calL(\vx;\lambda,\vy)$ is unique, for any $\epsilon>0$, there exists $\delta_0>0$ such that for all $\norm{\vx-\vx_{\lambda}}_2\geq \epsilon$,
    \begin{align*}
    \calL(\vx;\lambda)\geq\calL(\vx_\lambda;\lambda)+\delta_0.
    \end{align*}
    Let $\mathcal{B}(\vx,\epsilon)$ denotes the region $\{\vx_0\in \Omega:\norm{\vx_0-\vx}<\epsilon\}$. Actually, due to that $\Omega/\mathcal{B}(\vx_\lambda,\epsilon)$ is closed, $\min_{\vx\in\Omega/\mathcal{B}(\vx_\lambda,\epsilon)}\calL(\vx;\lambda) > \min_{\vx\in \Omega}\calL(\vx; \lambda) = \calL(\vx_\lambda;\lambda)$. Then any $\delta_0$ that satisfies $0<\delta_0 \leq \min_{\vx\in\Omega/\mathcal{B}(\vx_\lambda,\epsilon)}\calL(\vx;\lambda)-\calL(\vx_\lambda;\lambda)$ fulfill the requirements. 
    
    Let $\delta=\frac{\delta_0}{2C_3}$. Then for any $\overline{\lambda}$ that satisfies $\lambda-\delta<\overline{\lambda}<\lambda$, for any $\vx\in \Omega/\mathcal{B}(\vx_\lambda,\epsilon)$, it is obtained that
    \begin{equation*}
        \begin{aligned}
        \calL(\vx;\overline{\lambda})&\geq \calL(\vx;\lambda)-C_3(\lambda-\overline{\lambda})\\
   &   \geq \calL(\vx_\lambda;\lambda)+\delta_0 -\delta_0/2\\
  & \geq \calL(\vx_\lambda;\overline{\lambda})+\delta_0/2,
        \end{aligned}
    \end{equation*}
    where the monotonicity of minimization $\calL(\vx_;\lambda)$ for fixed $\vx$ with respect to $\lambda$ is used. 
    
    If $\lambda<\overline{\lambda}<\lambda+\delta$, following the same way, for all $\vx\in \Omega/\mathcal{B}(\vx_\lambda,\epsilon)$, it is obtained
     \begin{equation*}
        \begin{aligned}
        \calL(\vx;\overline{\lambda}) & \geq \calL(\vx;\lambda)\\
        & \geq \calL(\vx_\lambda;\lambda)+\delta_0 \\
  & \geq \calL(\vx_\lambda;\overline{\lambda})-C_3(\overline{\lambda}-\lambda)+\delta_0,\\
  &\geq \calL(\vx_\lambda;\overline{\lambda}) + \delta_0/2.
        \end{aligned}
    \end{equation*}
Anyway, when $|\overline{\lambda}-\lambda|<\delta$, for any $\vx\in \Omega/\mathcal{B}(\vx_\lambda;\epsilon)$, $\calL(\vx;\overline{\lambda})\geq \calL(\vx_\lambda;\overline{\lambda})+\delta_0/2$. Therefore, 
\begin{align*}
\min_{\vx\in \Omega}\calL(\vx;\overline{\lambda})\leq \min_{\vx\in \mathcal{B}(\vx_\lambda,\epsilon)}\calL(\vx;\overline{\lambda}) \leq \calL(\vx_\lambda;\overline{\lambda})< \min_{\vx\in \Omega/\mathcal{B}(\vx_\lambda, \epsilon)}\calL(\vx;\overline{\lambda}),
\end{align*}
which implies that the minimizer of $\calL(\vx;\overline{\lambda})$ is in $\mathcal{B}(\vx_\lambda,\epsilon)$, that is, $\vx_{\overline{\lambda}} \in \mathcal{B}(\vx_\lambda,\epsilon)$.

\item[(2)] If there are finite minimizers of $\calL(\vx;\lambda)$ that are denoted by $\vx_\lambda^i\ (i=1, 2, \dots, N)$, and the minimization is denoted by $\tilde{L}$, for any $\epsilon>0$, there exist $\delta_0$ that satisfies
\begin{align*}
0<\delta_0<\min_{\vx\in \Omega/( \cup_i\mathcal{B}(\vx_\lambda^i, \epsilon))}\calL(\vx;\lambda)-\tilde{L}
\end{align*}
such that for any $\vx\in \Omega/( \cup_i\mathcal{B}(\vx_\lambda^i,\epsilon))$, 
\begin{align*}
\calL(\vx;\lambda)\geq \tilde{L}+\delta_0.
\end{align*}
Similarly, it can be proved that there exists $\delta$ such that for any $|\lambda-\overline{\lambda}|<\delta$, the minimizers of $\calL(\vx;\overline{\lambda})$ satisfy $\vx_{\overline{\lambda}}\in \cup_i \mathcal{B}(\vx_\lambda^i, \epsilon)$.

\item[(3)] If there exist infinite minimizers of $\calL(\vx;\lambda)$, let the set of these minimizers is denoted by $A$ and the minimization is denoted by $\tilde{L}:=\calL(\vx;\lambda)\ (\vx\in A)$. Then due to the continuity of $\calL(\vx;\lambda)$ with respect to $\vx$, $A$ is closed. Define set $\mathcal{B}(A,\epsilon):=\{\vx\in \Omega:\norm{\vx_0-\vx}_2<\epsilon, \text{ for any }\vx_0\in A\}$. Then, since $\Omega/\mathcal{B}(A,\epsilon)$ is closed, $\min_{\vx\in \Omega/\mathcal{B}(A,\epsilon)} - \tilde{L}>0$ holds. And then for any $\epsilon>0$, there exists a positive $\delta_0$ that satisfies that 
\begin{align*}
\delta_0<\min_{\vx\in \Omega/\mathcal{B}(A,\epsilon)} \calL(\vx;\lambda) - \tilde{L}
\end{align*}
such that for all $\vx\in\Omega/\mathcal{B}(A,\epsilon)$, $\calL(\vx;\lambda)\geq \tilde{L}+\delta_0$. Similarly, it can be proved that there exists positive $\delta\leq\frac{\delta_0}{2C_3}$ such that for any $\overline{\lambda}$, if $|\overline{\lambda}-\lambda|<\delta$,  
\begin{align*}
\calL(\vx;\overline{\lambda})\geq\calL(\vx_A;\overline{\lambda})+\delta_0/2
\end{align*} 
holds for any $\vx\in \Omega/\mathcal{B}(A, \epsilon)$ and $\vx_A\in A$. It follows that
\begin{align*}
\min_{\vx\in \Omega}\calL(\vx, \overline{\lambda})\leq \min_{\vx\in \mathcal{B}(A,\epsilon)}\calL(\vx;\overline{\lambda}) \leq \min_{\vx_A\in \mathcal{B}(A, \epsilon)} \calL(\vx_A;\overline{\lambda}) < \min_{\vx\in \Omega/\mathcal{B}(A, \epsilon)} \calL(\vx;\overline{\lambda}).
\end{align*}
Therefore, the minimizers of $\calL(\vx;\overline{\lambda})$ lie in $\mathcal{B}(A,\epsilon)$, that is, for any $\vx_A\in A$, there exists a $\vx_{\overline{\lambda}}\in \argmin_{\vx\Omega}\calL(\vx;\overline{\lambda})$ such that $\norm{\vx_A-\vx_{\overline{\lambda}}}_2<\epsilon$.
\end{enumerate}
\end{proof}

\begin{theorem}
\label{the:exist-rx-lambda}
For any target image $\vx^*\in \Omega$ and noisy observation $\vy_\vec{\eta}=\vec{F}\vx^*+\vec{\eta}$, let Assumption 
\ref{as:noise-assumption} hold, then there exists a minimizer $\vx_\lambda^\vec{\eta}$ of $\calL(\vx;\lambda,\vy_\vec{\eta})$ in $\Omega$ with $\lambda>0$ such that $\vec{R}(\vx_\lambda^\vec{\eta})=\vec{R}(\vx^*)$.
\end{theorem}

\begin{proof}
According to Corollary \ref{cor:exist-rx}, there exist $\lambda_1$ and $\lambda_2$ such that any $\vx_1\in \argmin_{\vx\in \Omega}\calL(\vx;\lambda_1,\vy_\vec{\eta})$ and $\vx_2\in \argmin_{\vx\in \Omega}\calL(\vx;\lambda_2,\vy_\vec{\eta})$ satisfy
\begin{align*}
\vec{R}(\vx_1)\leq \vec{R}(\vx^*),\, \vec{R}(\vx_2)\geq \vec{R}(\vx^*).
\end{align*}
Define followings
\begin{align*}
\begin{aligned}
& l=\sup\{\lambda:\vec{R}(\vx_\lambda)> \vec{R}(\vx^*) \text{ for all } \vx_\lambda \in \argmin_{\vx\in \Omega}\calL(\vx;\lambda,\vy_{\vec{\eta}})\},\\
& u=\inf\{\lambda:\vec{R}(\vx_\lambda)< \vec{R}(\vx^*) \text{ for all }\vx_\lambda \in \argmin_{\vx\in \Omega}\calL(\vx;\lambda,\vy_{\vec{\eta}})\}.\\
\end{aligned}
\end{align*}
We can obtain $l\leq u$. Two cases are discussed separately.
\begin{enumerate}
\item[(1)]$l<u$.

Then for any $\lambda\in (l,u)$, according to Theorem \ref{the:monotonicity-of-R}, $\vec{R}(\vx_\lambda)\leq \vec{R}(\vx^*)$ and $\vec{R}(\vx_\lambda)\geq \vec{R}(\vx^*)$ hold for all minimizers $\vx_{\lambda}$ of $\calL(\vx; \lambda)$. Therefore, $\vec{R}(\vx_\lambda)=\vec{R}(\vx^*)$ holds for any $\lambda\in (l,u)$.
\item[(2)] $l=u$.

From the uniformly continuous of $\vec{R}(\vx)$ in $\Omega$, for any $\epsilon>0$, there exists $\delta$ such that when $\norm{\vx_1-\vx_2}_2< \delta$, $|\vec{R}(\vx_1)-\vec{R}(\vx_2)|<\epsilon$ holds. If $l=u$, according to Theorem \ref{the:continuous-lambda}, there exist $\delta_0$, such that for any $\overline{l}<l=u<\overline{u}$ that satisfies $|\overline{u}-\overline{l}|<\delta_0$, there exist $\vx_{\overline{l}}\in \argmin_{\vx\in \Omega}\calL(\vx;\overline{l},\vy_\vec{\eta})$ and $\vx_{\overline{u}}\in \argmin_{\vx\in \Omega}\calL(\vx;\overline{u},\vy_\vec{\eta})$ that satisfy 
\begin{align*}
\norm{\vx_{\overline{l}}-\vx_{\overline{u}}}_2<\delta.
\end{align*}
And then $|\vec{R}(\vx_{\overline{l}})-\vec{R}(\vx_{\overline{u}})|<\epsilon$. Moreover, due to the monotonicity of $\vec{R}(\vx_\lambda)$ with respect to $\lambda$, any minimizer $\vx_l\in \argmin_{\vx\in \Omega}\calL(\vx;l,\vy_\vec{\eta})$ satisfies that
\begin{align*}
 \vec{R}(\vx_{\overline{u}})\leq \vec{R}(\vx_l)\leq \vec{R}(\vx_{\overline{l}}).
\end{align*}
And combining $\vec{R}(\vx_{\overline{u}})< \vec{R}(\vx^*)<\vec{R}(\vx_{\overline{l}})$, we can obtain that  
\begin{align*}
|\vec{R}(\vx_l)-\vec{R}(\vx^*)| <  |\vec{R}(\vx_{\overline{u}})-\vec{R}(\vx_{\overline{l}})|<\epsilon.
\end{align*}
Therefore, $\vec{R}(\vx_l)=\vec{R}(\vx^*)$.
\end{enumerate}

\end{proof}

In reality, the true $\pi_{\theta_\star}(\vx^*)$ is unknown, so we can use training result $\frac{1}{|\calD|}\sum_{\vx\in \calD}\pi_{\theta_\star}(\vx)$ as an approximation of $\pi_{\theta_\star}(\vx^*)$. According to the proof of Theorem \ref{the:exist-rx-lambda}, the closed interval method can be used to find a $\lambda$ that makes the regularization term $\vec{R}(\vx_{\lambda})$ approximate $\frac{1}{|\mathcal{D}|}\sum_{\vx\in \mathcal{D}}-\log \pi_{\theta_\star}(\vx^*)$.

\section{Numerical implementation}
\label{sec:4}

In this section, to implement the rule \eqref{eq:choice-lambda}, we present an incremental gradient descent algorithm with adaptive regularization parameters to address the estimation of $\vx$ from the minimization of \eqref{eq:obj-glow}. Moreover, a patch-based method is applied to address the reconstruction of larger-size images than training data, see \ref{sec:4.2}.

\subsection{Incremental gradient descent algorithm}
\label{sec:4.1}
As we discussed previously, we use 
\begin{equation}
    \label{eq:prior-ref}
    C_{\mathcal{D}, \theta^\star} := \frac{1}{|\mathcal{D}|}\sum_{\vx\in \mathcal{D}}-\log \pi_{\theta_\star}(\vx^*)
\end{equation}
in the rule \eqref{eq:choice-lambda} instead of unknown $\vec{R}(\vx^*)$ in practice. 
For fixed $\lambda$, the optimization problem \eqref{eq:3.4} can be solved by the incremental gradient descent method, which alternately updates the data fidelity and regularization terms at each iteration~\cite{bertsekas1997nonlinear}. In addition, combining the closed interval method to update $\lambda$, an algorithm to minimize \eqref{eq:3.4} with adaptive regularization parameter is presented in \ref{alg:1}.

\begin{algorithm}[!htb]
\caption{Incremental gradient descent with adaptive regularization \label{alg:1}}
\begin{algorithmic}[1]
  \STATE Given fixed \Glow model with the parameters $\theta_\star$ and the prior $C_{\mathcal{D},\theta^\star}$, initial lower and upper bound of $\lambda$: $l_0$ and $u_0$, $\lambda_0 = l$, initial guess $\vx^0$, tolerance $\epsilon_1$ and $\epsilon_2$, rate $\beta\in (0,1)$.
  \FOR{$i = 0, 1, \dots$}
  \STATE $\vz_0 = \vx^i$
  \FOR {$j = 1, 2, \dots$}
  \STATE
  $\tilde{\vz} = \vz_{j-1}-s\nabla(\frac{1}{2}\norm{\vec{F}\vz_{j-1}-\vy_\vec{\eta}}_2^2)$\\
  $\vz_{j} = \tilde{\vz}-t \lambda_i\nabla \vec{R}(\tilde{\vz})$
  \ENDFOR
  \STATE $\vx^{i+1}=\vz_j$
  \STATE Stop if one of these conditions are met: (i) $|\vec{R}(\vx^{i+1})-C|\leq\epsilon_1$; (ii) $|u_i-l_i|\leq\epsilon_2$;\\
  Otherwise, execute one of the following steps,
    \begin{enumerate}
      \item[(i)] if $\vec{R}(\vx^{i+1})<C$, let $l_{i+1}=l_i$, $u_{i+1}=\lambda_i$, and $\lambda_{i+1}=\lambda_i-\beta |u_{i+1}-l_{i+1}|$;
      \item[(ii)] if $\vec{R}(\vx^{i+1})>C$, let $l_{i+1}=\lambda_i$, $u_{i+1}=u_i$, and $\lambda_{i+1}=\lambda_i+\beta |u_{i+1}-l_{i+1}|$.
  \end{enumerate}
 \ENDFOR
\end{algorithmic}
\end{algorithm}

\begin{remark}
In algorithm \ref{alg:1}, it is assumed that $l_0$ and $u_0$ are lower and upper bound of $\lambda$, which means that the reconstruction $\vx_{l_0}^\vec{\eta}$ and $\vx_{u_0}^\vec{\eta}$ satisfies 
\begin{equation}
    \label{eq:l0-u0}
    \vec{R}(\vx_{l_0})\geq C_{\mathcal{D}, \theta^\star} \geq \vec{R}(\vx_{u_0}).
\end{equation}
Then from the initial lower and upper bound of $\lambda$, the closed interval between them keeps decreasing, and the interval size satisfies $|u_{i+1}-l_{i+1}|\leq (1-\beta)|u_i-l_i|$ or $|u_{i+1}-l_{i+1}|\leq \beta|u_i-l_i|$. Therefore, after enough iterations, $|u_i-l_i|$ or $|\vec{R}(x^i)-C_{\mathcal{D}, \theta^\star}|$ will meet the tolerance. 

The lower and upper bound of $\lambda$ can be precomputed. Updating $\lambda$ using $\lambda=\beta \lambda \ (\beta>1)$ and comparing the corresponding $\vec{R}(\vx_{\lambda})$ and $C_{\mathcal{D}, \theta^\star}$ can find $\lambda$ such that $\vec{R}(\vx_{\lambda})\leq C_{\mathcal{D}, \theta^\star}$. On the contrast, using $\lambda=\lambda \beta\ (\beta<1)$ can find $\lambda$ such that $\vec{R}(\vx_{\lambda})\geq C_{\mathcal{D}, \theta^\star}$. 

The function called Autograd in deep learning packages, such as Pytorch helps compute the gradient of the regularization term.
\end{remark}

As it is analyzed in Section \ref{sec:PAT.IP}, to reconstruct the target image $\vx^*$, if the Gaussian noise level $\norm{\vec{\eta}}_2$ is known, the regularization parameter should be $\lambda=\norm{\vec{\eta}}_2$. However, when the noise level is unknown, a parameter $\lambda$ such that $\vec{R}(\vx^{\vec{\eta}}_\lambda)$ matches the exact $\vec{R}(\vx^*)$ is an admissible choice. Normally, $C_{\mathcal{D}, \theta^\star}$ that is from training data is used as an approximation of $\vec{R}(\vx^*)$. Therefore, if the exact $\vec{R}(\vx^*)$ is far away from the former, we can relax the expected range of $\vec{R}(\vx_\lambda^{\vec{\eta}})$ such that $\vec{R}(\vx_{\lambda}^{\vec{\eta}})\in (\vec{R}(\vx^*)-a, \vec{R}(\vx^*)+a)$ for some relatively large $a>0$, in which only approximated $\vec{R}(\vx^*)$ is needed, so the tolerance in Algorithm \ref{alg:1} don't need to be small.

\subsection{Patch-based regularization }
\label{sec:4.2}

The cost of training deep learning methods increases significantly with larger image sizes. This is particularly true for medical images, which are often three-dimensional. Therefore, we apply a patched-based regularization method~\cite{Knyaz2019} that only considers randomly selected patches of the same scale as the one used during the training phase. If the size of training images is $N_x\times N_y\times N_z$, to estimate the image of size $M_x\times M_y\times M_z$, the minimization objective is written as 
\begin{equation}
  \label{eq:3.9}
  \calL(\vx;\lambda, \vy_n):=\frac{1}{2}\norm{\vec{F}\vx-\vy_n}_2^2 - \frac{\lambda}{M}\sum_{m=1}^M\log (\pi_{\theta_\star}(\vx[i_m, j_m, k_m;N_x, N_y, N_z]))
\end{equation}
where $\vx[i_m, j_m, k_m;N_x, N_y, N_z]$ denotes the randomly selected three-dimensional submatrix of $\vx$ of scale $N_x\times N_y\times N_z$, that is,
\begin{equation}
  \label{eq:3.10}
  \vx[i_m, j_m, k_m;N_x, N_y, N_z] = \vx(i_m:i_m+N_x-1, j_m:j_m+N_y-1, k_m:k_m+N_z-1).
\end{equation}
For each iteration, $M$ patches are randomly selected with the starting point of index $(i_m,j_m,k_m)$ in the discrete grids.

\section{Simulations}
\label{sec:5}

\paragraph{Dataset and training}
In this section, a lung dataset “ELCAP Public Lung Image Database” which consists of 50 whole-lung \ac{CT} scans~\cite{lungdataset} is used to verify the effectiveness and efficiency of the \ac{NFR}. Training data consists of 2400 images with size $32\times 32\times 32$ selected from the last 40 patients. Validation data are selected from the first 10 patients. Before implementing training and reconstruction, Frangi vesselness filter is applied to obtain vessel-like structures and reduce background noise~\cite{Frangi1998}. To increase the robustness of the network when only a few samples are available, data augmentation is an essential tool to increase the diversity of the dataset by applying random transformations, such as rotation, shift, and elastic deformations. For our three-dimensional images, a Python library TorchIO \cite{perez-garcia_torchio_2021} is applied, in which random affine transform with scale \numrange{0.8}{1.2}, rotation with degree \numrange{0}{90}, flip and shift are considered. A three-layer \Glow model is trained where each layer consists of $7$ steps of flow with maximum hidden channels of $512$. 

\paragraph{Measurement system}

After training the \Glow model, we apply the image prior to the reconstruction. Gaussian noise $\vec{n}$ with level $\sigma$ to the clear data is added as follows,
\begin{equation}
  \label{eq:4.1}
  \vy_n = \vy + \vec{n}, \vec{n}\sim\normal(\mathbf{0}, (\sigma * \max \vy)^2).
\end{equation}

Following clinical specifications of limited-view measurement, transducer arrays in our simulations are positioned on a hemisphere as shown in Figure \ref{fig:system}. To investigate the efficacy of NFR in sparse measurements, we have implemented diverse sparse levels to produce measurements. These levels involve the deployment of transducer arrays uniformly distributed across distinct azimuth and polar angles in a hemisphere. To achieve this, we have implemented varying levels of sparsity, that is to use specific combinations of azimuth and polar angles $(256, 16)$, $(128, 12)$, and $(64, 8)$, respectively. To evaluate the robustness to noise level, 0.01 and 0.05 Gaussian noise have been incorporated into the measurements, as detailed in Equation \eqref{eq:4.1}.
 
\begin{figure}[!htb]
    \centering
    \includegraphics[trim=200 200 200 30,clip,width=0.8\textwidth]{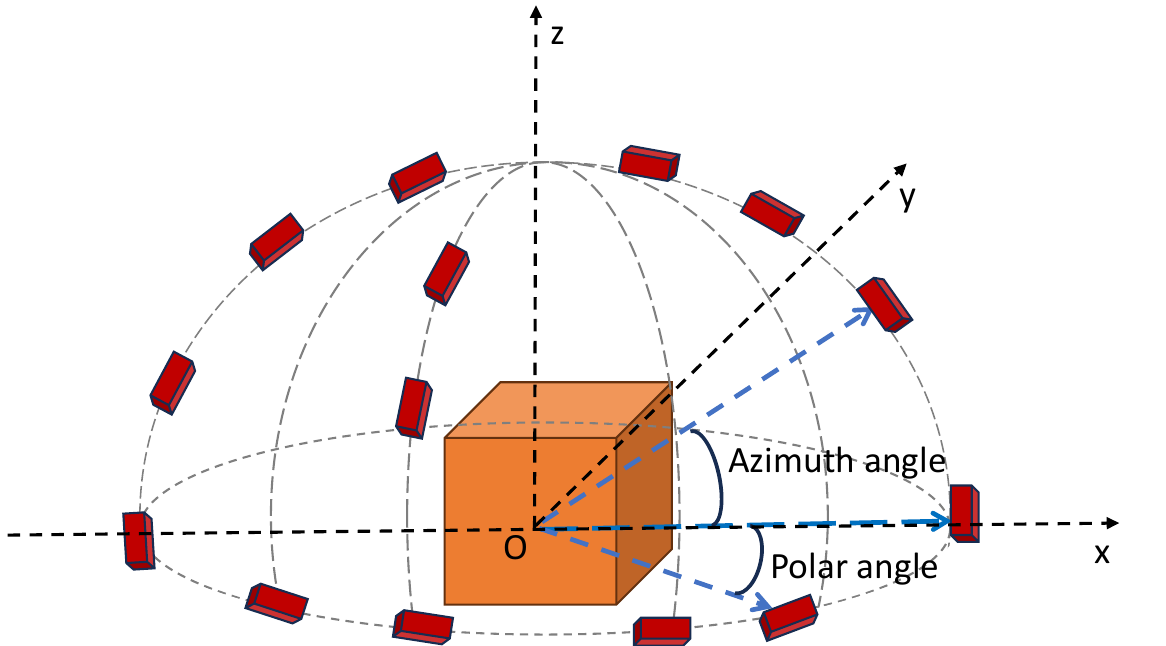}
    \caption{\small{\label{fig:system} Measurement system schematic. The reconstructed object (marked as the orange cube) is placed in the center of the coordinate system. The transducers (marked as red blocks) are placed on a hemispherical surface and arranged according to the azimuth ($0^\circ-360^\circ$) and polar ($0^\circ-90^\circ$) angles.}}
    \label{fig:enter-label}
\end{figure}

\paragraph{Evaluation}

Considering our focus on the contrast of the image, as detailed in~\cite{Hauptmann2018, jin2017deep}, for the reconstruction $\overline{\vx}$, the scaled and unbiased reconstruction is derived from
\begin{equation*}
    (\vx_r, a, b) = \argmin_{\vx,a,b} \norm{\vx - a*\overline{\vx}-b}_2,
\end{equation*}
where $a, b \in \mathbb{R}$.
To assess the quality of image reconstruction, metrics including  Structural Similarity Index (\ac{SSIM})~\cite{ssim}, Peak Signal-to-Noise Ratio (\ac{PSNR}), and relative reconstruction accuracy (\ac{RRA}) of the scaled reconstructions are computed, in which \ac{RRA} is defined as
\begin{equation}
  \label{eq:4.2}
  \text{\ac{RRA}}(\vx_r) = \frac{\norm{\vx_r-\vx^*}_2}{\norm{\vx^*}_2}.
\end{equation}
for ground truth $\vx^*$.

In Section \ref{sec:5.1}, the effect of the regularization parameter on \ac{RRA} and the regularization term is explored with different sparsity levels and noise levels, in which the image to be reconstructed is segmented from the first 10 patients and has the same size with training data. In Section \ref{sec:5.2}, \ac{NFR} method is applied to a larger image of size $128 \times 128 \times 128$ that is segmented from the first 10 patients. Furthermore, a comparative analysis is performed between \ac{NFR} and two state-of-the-art image reconstruction approaches, post-processing \Unet and \ac{TV} regularization, in the case of different sparsity levels and noise levels.

\subsection{Effect of regularization parameter}
\label{sec:5.1}

In this section, \ac{NFR} method is applied to an image selected from the validation data of the same size with training data, $32\times 32\times 32$, to study the effect of the regularization parameter on the reconstruction results. Various sparsity levels, characterized by the number of azimuth and polar angles $(64, 8)$, $(128, 12)$ and $(256, 16)$, and various noise levels of 0.01 and 0.05, are considered. The minimization of the objective function \eqref{eq:obj-glow} is performed for different fixed values of $\lambda$. Only the reconstruction results with the smallest \ac{RRA} from the data with 0.05 noise are shown in Figure \ref{fig:4.1}. Naturally, the more measurements the better results, and all the maximum values of error images are below 0.15. In Figure \ref{fig:4.2}, the \ac{RRA} and regularization term $\vec{R}(\vx)=-\log \pi_{\theta_\star}(\vx)$ are plotted to the regularization parameter under each measurement scenario. Additionally, the reference value $\frac{1}{|\calD|}\sum_{\vx\in \calD}-\log \pi_{\theta_\star}(\vx)=2.34$ that is obtained from training data is plotted as well. As we analyzed in Section \ref{sec:3}, the regularization value $\vec{R}(\vx)$ is monotonically decreasing with respect to $\lambda$. The numerical results align closely with our analysis within the accuracy range. Furthermore, it shows that the $\vx$ that minimizes \ac{RRA} is also the $\vx$ that brings $\vec{R}(\vx)$ closest to our expected reference value of 2.34.

\begin{figure}[!htb]
  \centering
  \begin{minipage}[b]{.24\textwidth}
      \includegraphics[trim=0 50 30 0,clip,width=\textwidth]{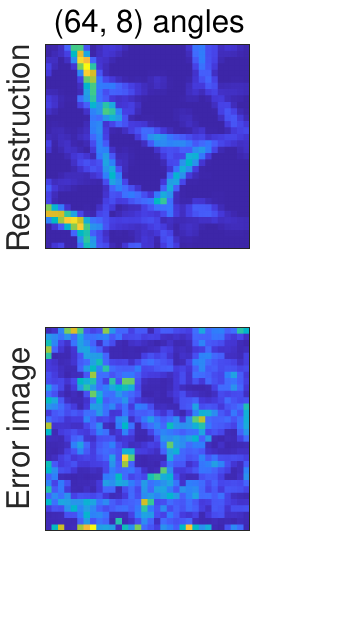}
  \end{minipage}
  \begin{minipage}[b]{.24\textwidth}
      \includegraphics[trim=30 50 30 0,clip,width=.93\textwidth]{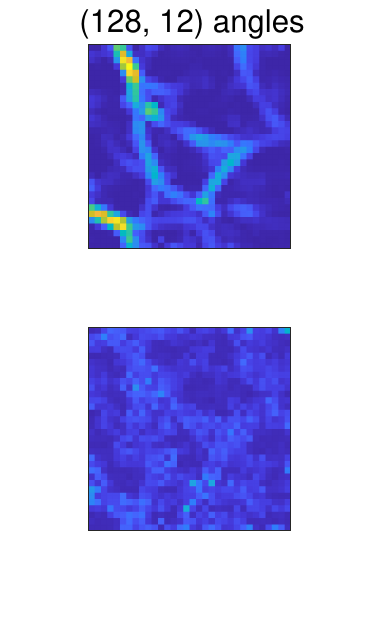}
  \end{minipage}
  \begin{minipage}[b]{.24\textwidth}
      \includegraphics[trim=38 50 0 0,clip,width=1.095\textwidth]{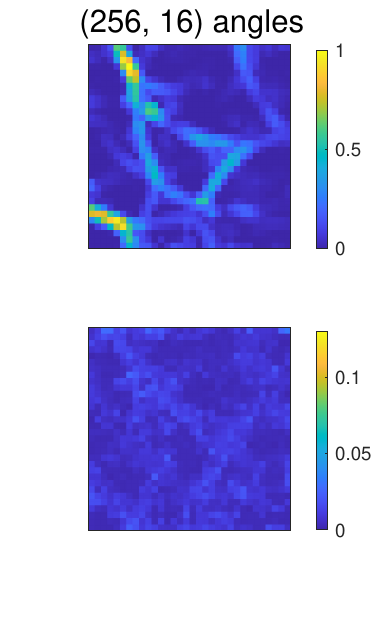}
  \end{minipage}
      \caption{\small{\label{fig:4.1}Maximum projection intensity (MIP) of the scaled reconstruction results $\vx_r$ and error images $\left|\vx_r-\vx^{true}\right|$ of the small phantom with size $32\times 32 \times 32$ from the data with measurements of different sparse levels and 0.05 Gaussian noise.}}
\end{figure}

\begin{figure}[!htb]
    \centering
    \subfloat
    {\raisebox{-0.75in}{\rotatebox[origin=t]{90}{(64, 8) angles}}}  
    \quad
    \subfloat[Noise level 0.01]
    {\includegraphics[trim=0 0 0 0,clip,width=0.4\textwidth]{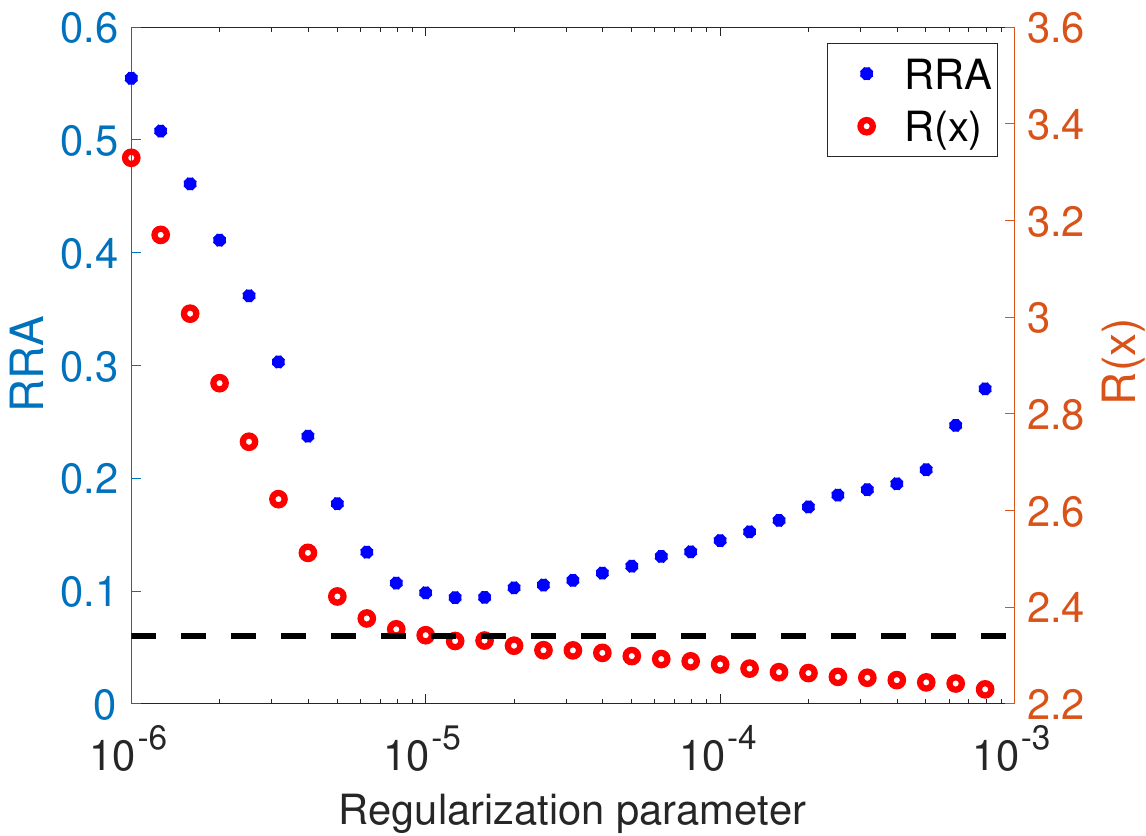}}
    \subfloat[Noise level 0.05]
    {\includegraphics[trim=0 0 0 0,clip,width=0.4\textwidth]{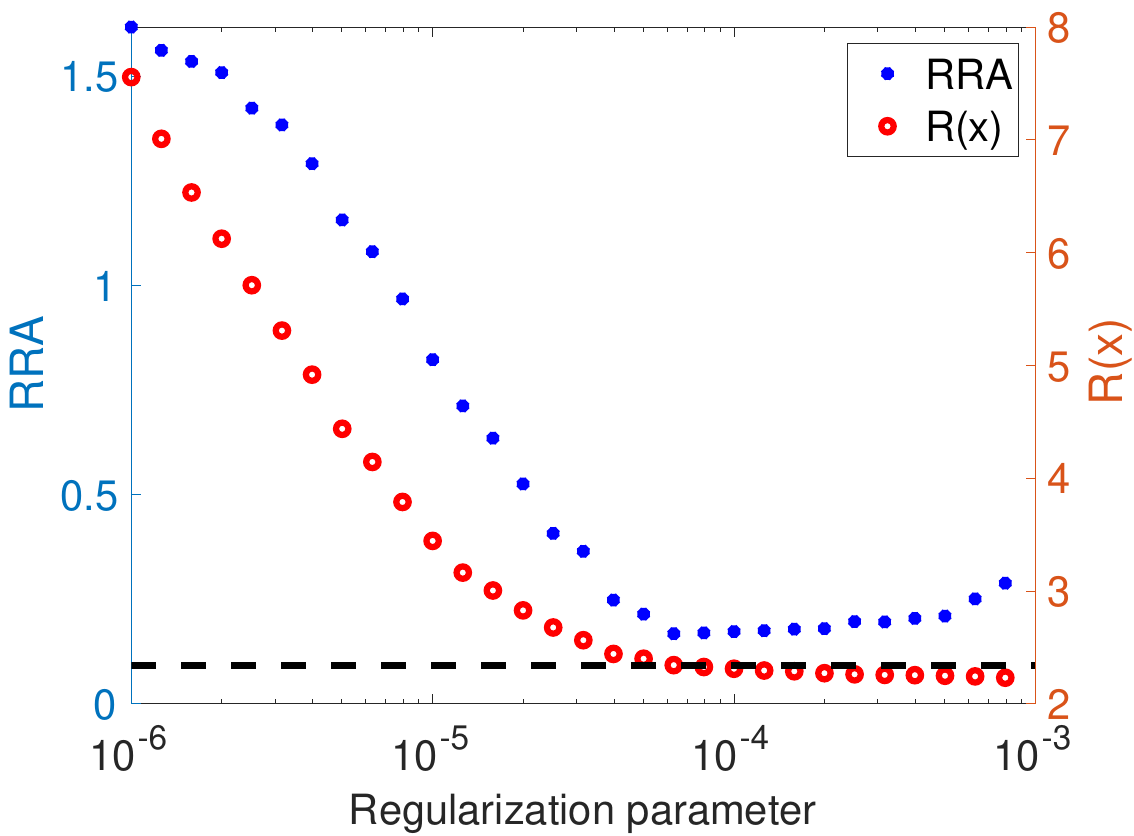}}
    \\
    {\raisebox{-0.75in}{\rotatebox[origin=t]{90}{(128, 12) angles}}}   
    \quad
    \subfloat
    {\includegraphics[trim=0 0 0 0,clip,width=0.4\textwidth]{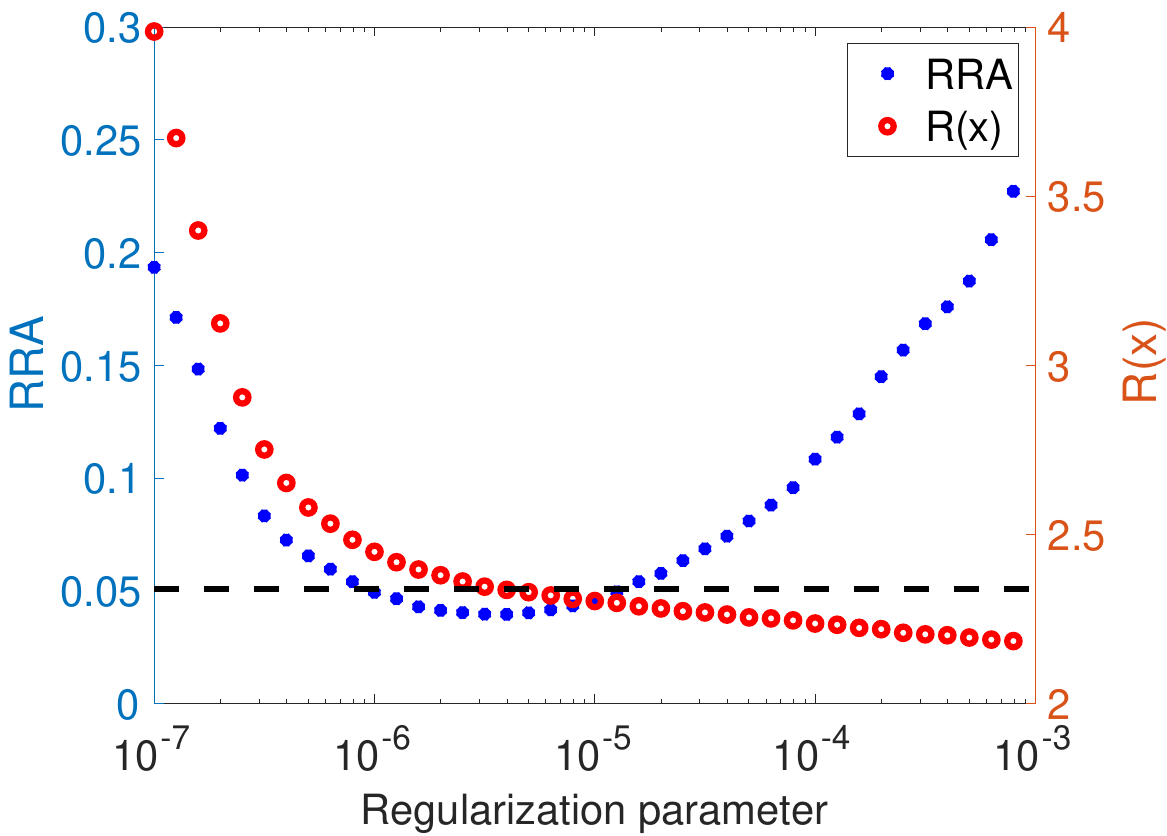}}
    \subfloat
    {\includegraphics[trim=0 0 0 0,clip,width=0.4\textwidth]{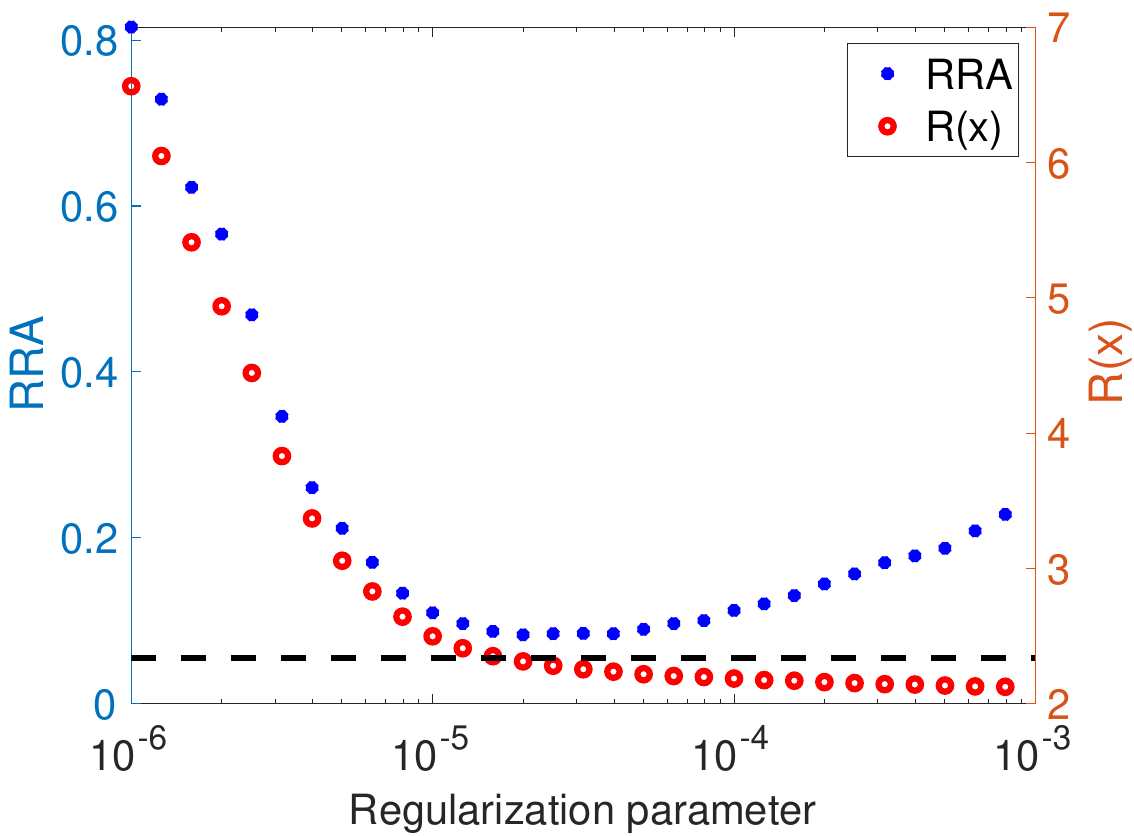}}
    \\
    {\raisebox{-0.75in}{\rotatebox[origin=t]{90}{(256, 16) angles}}}   
    \quad
    \subfloat
    {\includegraphics[trim=0 0 0 0,clip,width=0.4\textwidth]{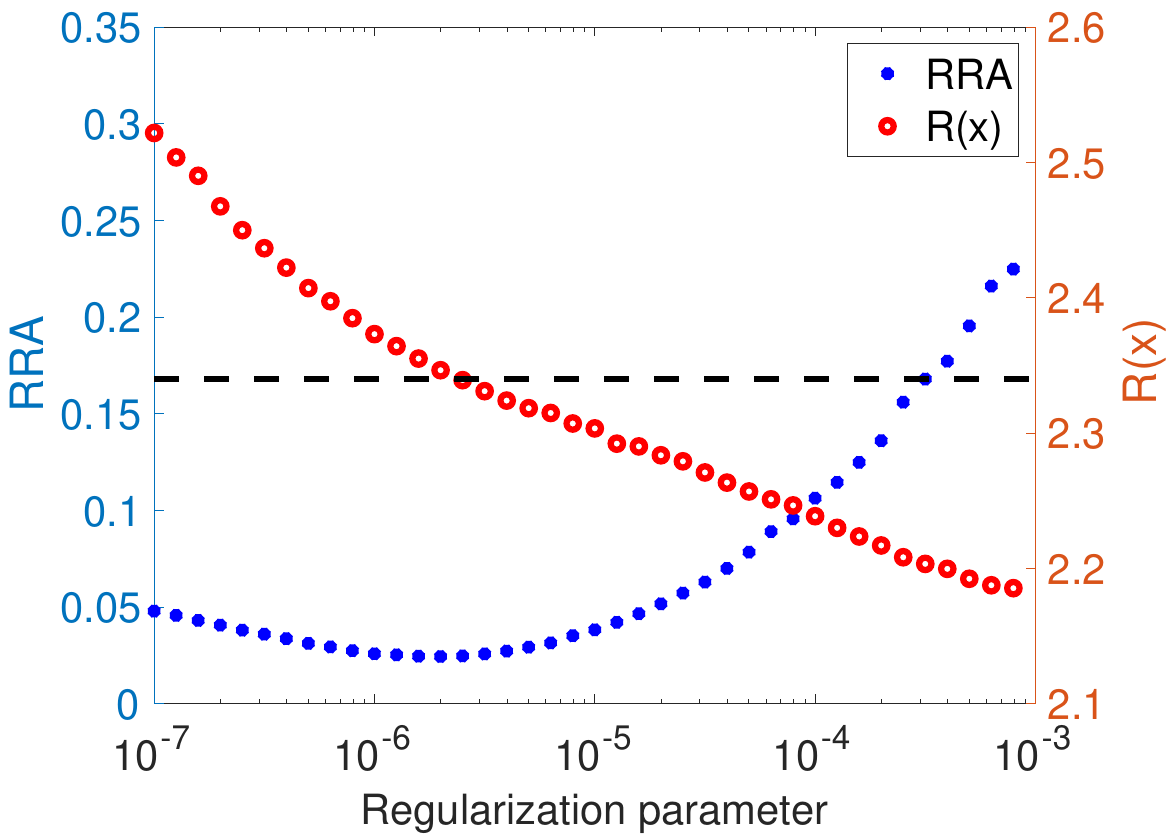}}
    \subfloat
    {\includegraphics[trim=0 0 0 0,clip,width=0.4\textwidth]{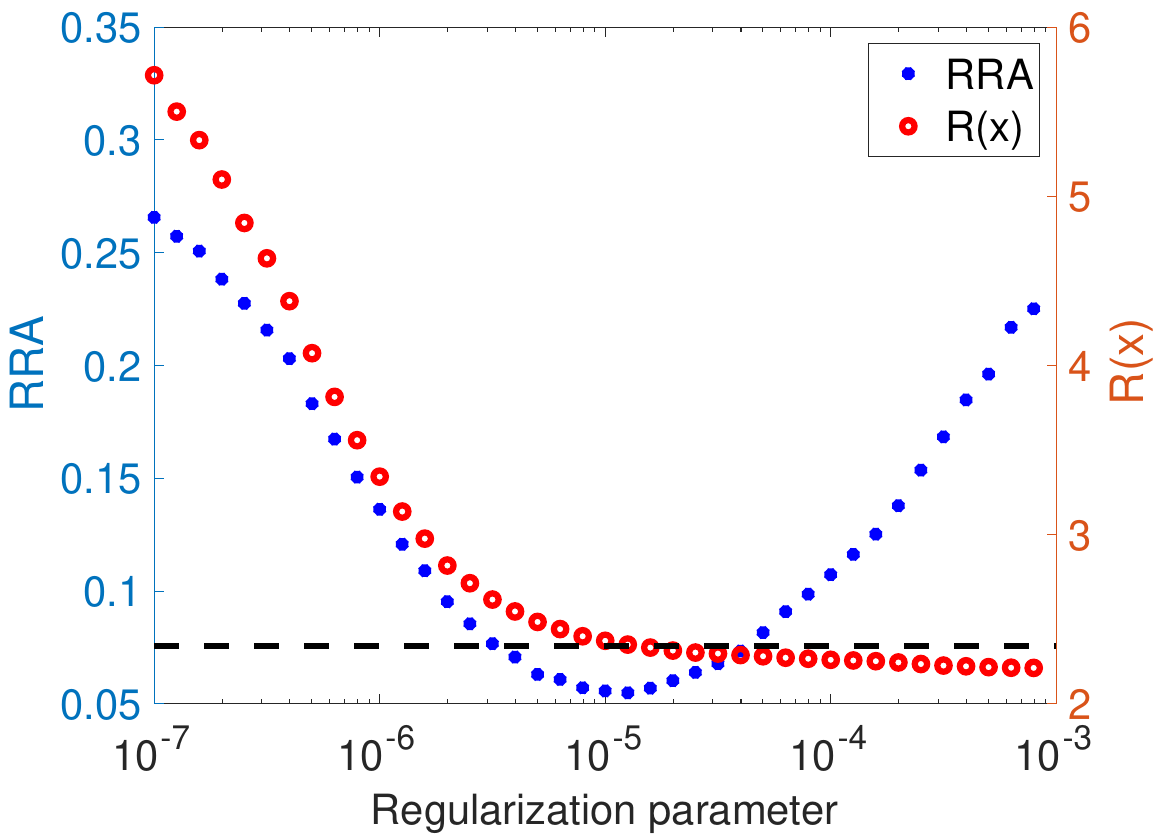}}
    \caption{\small{\label{fig:4.2} Comparison of \ac{RRA} (blue dots) and regularization term (red circles) of reconstruction results of the phantom with size $32\times 32\times 32$ with different measurements and noise level. The black dotted line indicates the reference to the regularization term 2.34. The rows represent measurements with the number of azimuth angles and elevation angles (64, 8), (128, 12), and (256, 16), and the columns represent noise levels 0.01, and 0.05 respectively.}}
\end{figure}

\subsection{\ac{NFR} method application}
\label{sec:5.2}

To demonstrate \ac{NFR} method on large-scale images, we choose a validation lung image of size $128\times 128\times 128$ segmented from the first 10 patients, as shown in Figure \ref{fig:true}. To avoid the inverse crime, the measurements are obtained by implementing the forward operator on the scaled image, enlarged by a factor of 2, resulting in dimensions of $256 \times 256 \times 256$. Subsequently, Algorithm \ref{alg:1} is applied to the data with varying configurations of azimuth and polar angles, $(256, 16)$, $(128, 12)$ and $(64, 8)$, and 0.05 Gaussian noise. The reconstruction results and their corresponding error images are presented in Figure \ref{fig:4.3}. It is observed that an increased number of azimuth and polar angles contributes to the retrieval of finer structures with enhanced quality. 

\begin{figure}[!htb]
  \centering
  \includegraphics[trim=0 0 0 0,clip,width=\textwidth]{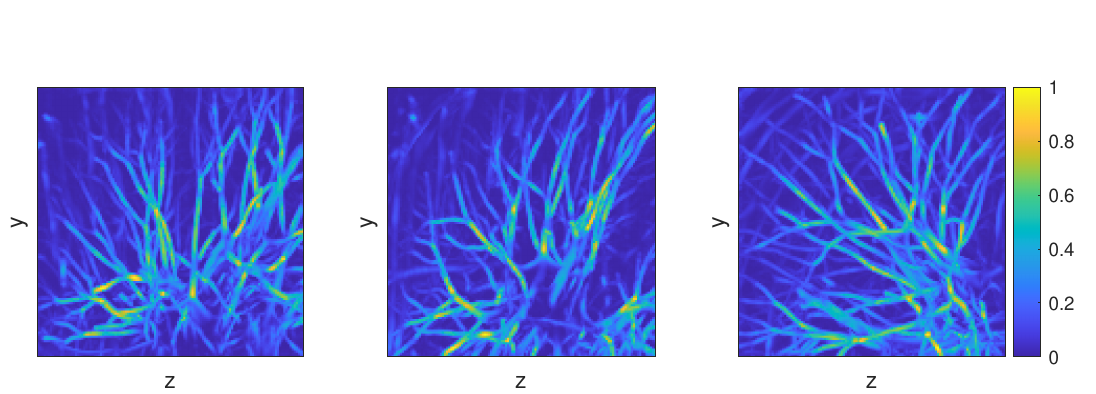}
  \caption{\small{\label{fig:true}Maximum projection intensity (MIP) of validation image along x (left), y (middle), and z-axis (right).}}
\end{figure}

\begin{figure}[!htb]
  \centering
  \includegraphics[trim=0 0 0 0,clip,width=\textwidth]{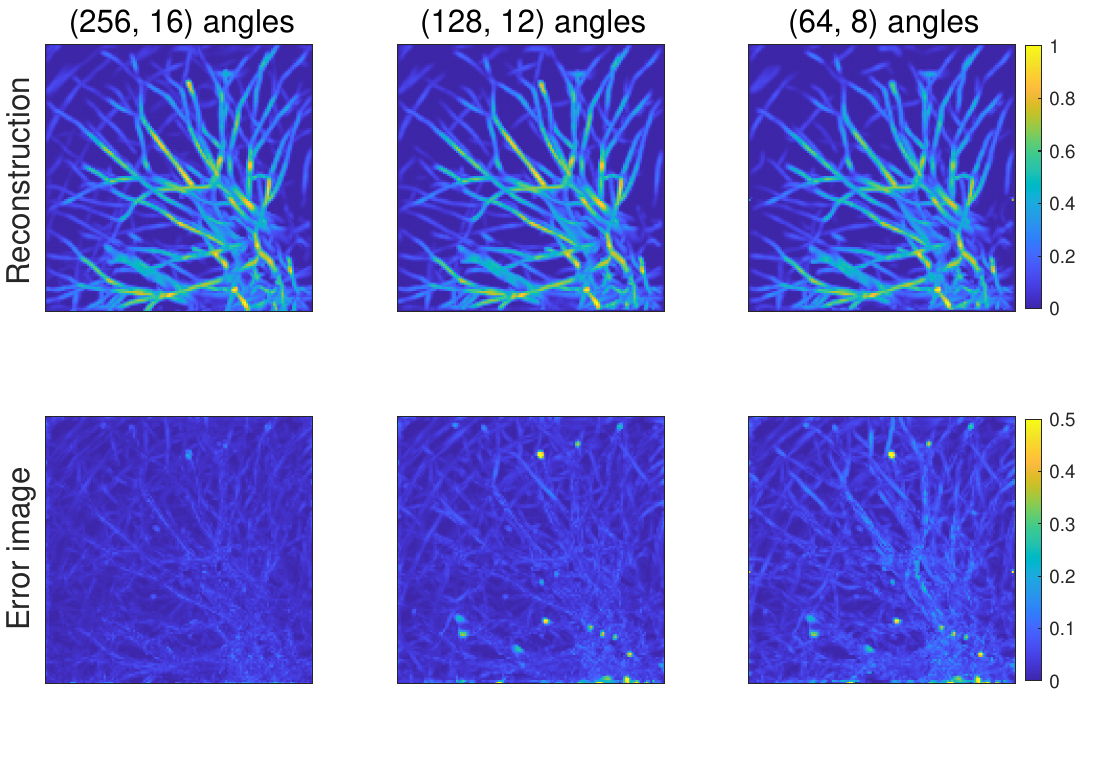}
  \caption{\small{\label{fig:4.3}Maximum projection intensity (MIP) along the z-axis of the reconstruction results $\vx_r$ and error images $\left|\vx_r-\vx^{true}\right|$ by \ac{NFR} method from the data with the different number of azimuth and polar angles $(256, 16)$, $(128, 12)$ and $(64, 8)$ and 0.05 Gaussian noise.}}
\end{figure}

\subsection{Comparison of \ac{NFR} method,  \ac{TV} regularization and post-processing deep learning}
\label{sec:5.3}

The total variation (\ac{TV}), a classical regularization in image reconstruction, is widely used to enhance image quality, particularly in the domain of medical imaging with incomplete and noisy measurements~\cite{wang2008new, beck2009}. \ac{TV} assumes that target images have sharp edges and a smooth background. Therefore, it usually provides a valuable benchmark to assess various reconstruction methods. \Unet as a convolution neural network (\ac{CNN}), achieved significant improvements in deep learning for effectively addressing imaging problems across standard datasets. Serving as a post-processing technique, the \Unet architecture typically involves training a model to effectively map a rough reconstruction or an image derived from measurement to a high-quality reconstruction. Leveraging extensive paired data and a large number of model parameters, \Unet outperforms numerous traditional methods. In this study, we utilize the \Unet architecture with skip connections introduced in \cite{Ronneberger2015} to a three-dimensional version and train a model to map the adjoint of measurement data, $\vec{F}^*\vy_{\vec{n}}$, to the true images $x^*$. It should be emphasized that to avoid the so-called inverse crime, the measurement data of the TV method, NFR method, and $\vy_{\vec{n}}$ in \texttt{U-Net} are obtained by implementing the forward operator on the scaled image of size $256\times 256\times 256$. To investigate the effect of noise and sparse levels in the measurement, the following configurations are made and corresponding regularization parameters are reported:
\begin{enumerate}
    \item Sparsity effect: fixed noise level 0.05, with varying numbers of the azimuth and polar angles $(256, 16)$, $(128, 12)$ and $(64, 8)$. Corresponding regularization parameters for \ac{TV} method are all 3e-5; for \ac{NFR} are 0.0022, 0.0036, and 0.0041.
    \item Noise effect: varying noise 0.05 and 0.01, with a fixed number of the azimuth and polar angles $(64, 8)$. Corresponding regularization parameters for \ac{TV} method are 3e-5 and 2e-5; for \ac{NFR} are 0.0041 and 0.0023.
\end{enumerate}
As a comparison, the sensitivity of \Unet to the measurement setting and data quantity is evaluated through three training configurations. In \Unet$\!\!^1$ and \Unet$\!\!^2$, the networks are trained for each specific noise levels and measurement settings; in \Unet$\!\!^3$, the network is trained across all noise levels and measurement settings:
\begin{enumerate}
\item \Unet$\!\!^1$: Trained with 640 pairs of data.
\item \Unet$\!\!^2$: Trained with 2400 pairs of data.
\item \Unet$\!\!^3$: Trained with 2400 pairs of data.
\end{enumerate}


\paragraph{Sparsity effect}
Table \ref{tab:1} presents the comprehensive quantitative comparison of the reconstruction results from the measurement with fixed 0.05 Gaussian noise and varying sparsity levels. Figure \ref{fig:4.4} and \ref{fig:4.5} show the reconstructed and error images with the number of azimuth and polar angles $(64, 8)$ and $(256, 16)$, respectively. The regularization parameter of \ac{TV} method is tuned to obtain the smallest \ac{RRA}. Table \ref{tab:1} shows that \ac{NFR} achieves the best reconstruction quantitatively. \ac{TV} comes second, and \Unet is the worst. From the reconstructed and error images, it is seen that \ac{NFR} consistently outperforms others. It exhibits enhanced detailed recovery, a smoother background, and the smallest error images. Conversely, the value of the error image from \Unet is the highest, which has a lot of noise in the background. However, compared with the staircase artifacts of blood vessels reconstructed by \ac{TV}, \Unet and \ac{NFR} exhibit a more realistic shape. This observation also illustrates the capacity of the convolutional network to capture and extract intricate image features. 
%
With an increase in the amount of training data, \Unet$\!\!^2$ demonstrates superior performance compared to \Unet$\!\!^1$. The relatively poor performance of \Unet$\!\!^3$ indicates that \Unet, as a supervised neural network methodology, does not generalize effectively. However, \Unet remains a valuable benchmark for comparative analysis. The superior performance of \ac{NFR} further highlights its advantages, particularly in scenarios where data availability is limited.
%

 \begin{table}[!htb]
   \caption{\label{tab:1}Quantitative comparison of the scaled reconstruction results by \ac{NFR} method, \ac{TV} method, and \Unet post-processing from measurement data with 0.05 Gaussian noise and different sparsity levels.}
  \centering
  \begin{tabu}{l|l||ccc}
    \tabucline[2pt]{-}
     \multicolumn{2}{c}{Angles}   & (256, 16)  &  (128, 12)  &  (64, 8)  \\
    \hline\hline
    \multirow{3}{*}{\ac{PSNR}}& \ac{NFR} &\bf{47.54}&\bf{43.63}& \bf{41.38}  \\
                         & \ac{TV}  & 44.48 & 41.58   & 39.36      \\
                         & \Unet$\!\!^1$       &  40.36& 37.73   & 34.78      \\
                         & \Unet$\!\!^2$  &42.52 & 40.42  & 36.73\\
                         & \Unet$\!\!^3$  &38.21 & 35.82  & 32.21\\
    \hline
    \multirow{3}{*}{\ac{SSIM}}& \ac{NFR}  & \bf{0.98}  & \bf{0.97}    & \bf{0.97}         \\
                         & \ac{TV}   & 0.97 & 0.96    &  0.95       \\
                         & \Unet$\!\!^1$       &  0.96 & 0.95    &  0.89       \\
                         & \Unet$\!\!^2$       &0.97   & 0.95    &   0.92      \\
                         & \Unet$\!\!^3$       &0.93   & 0.91    &   0.88      \\
    \hline
     \multirow{3}{*}{\ac{RRA}}& \ac{NFR}  & \bf{0.13}& \bf{0.21} &\bf{0.27}  \\
                         & \ac{TV}   &  0.19 &  0.26  &   0.34    \\
                         & \Unet$\!\!^1$       & 0.30 &  0.41   &  0.57     \\
                         & \Unet$\!\!^2$       &0.24   & 0.30    &  0.46       \\
                         & \Unet$\!\!^3$       &0.33   & 0.45    &  0.59      \\
    \tabucline[2pt]{-}
  \end{tabu}
\end{table}

\begin{figure}[!htb]
  \centering
  \includegraphics[trim=0 0 0 0,clip,width=\textwidth]{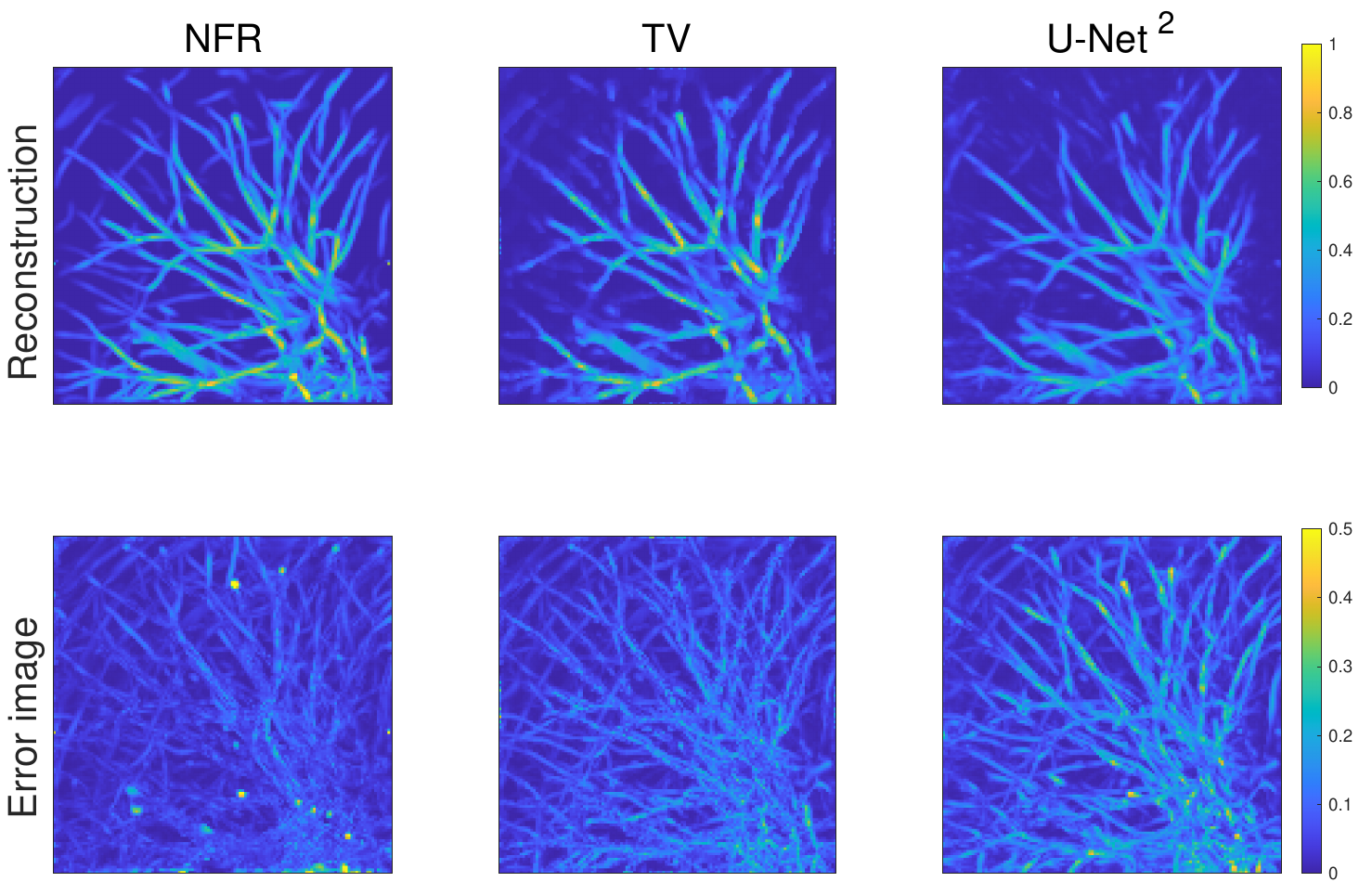}
  \caption{\small{\label{fig:4.4}Maximum projection intensity (MIP) along z-axis of the scaled reconstruction results $\vx_r$ and error images $\left|\vx_r-\vx^{true}\right|$ by \ac{NFR} method, \ac{TV} regularization and \Unet from the data with the number of azimuth and polar angles $(64, 8)$ and 0.05 Gaussian noise.}}
\end{figure}

\begin{figure}[!htb]
  \centering
  \includegraphics[trim=0 0 0 0,clip,width=\textwidth]{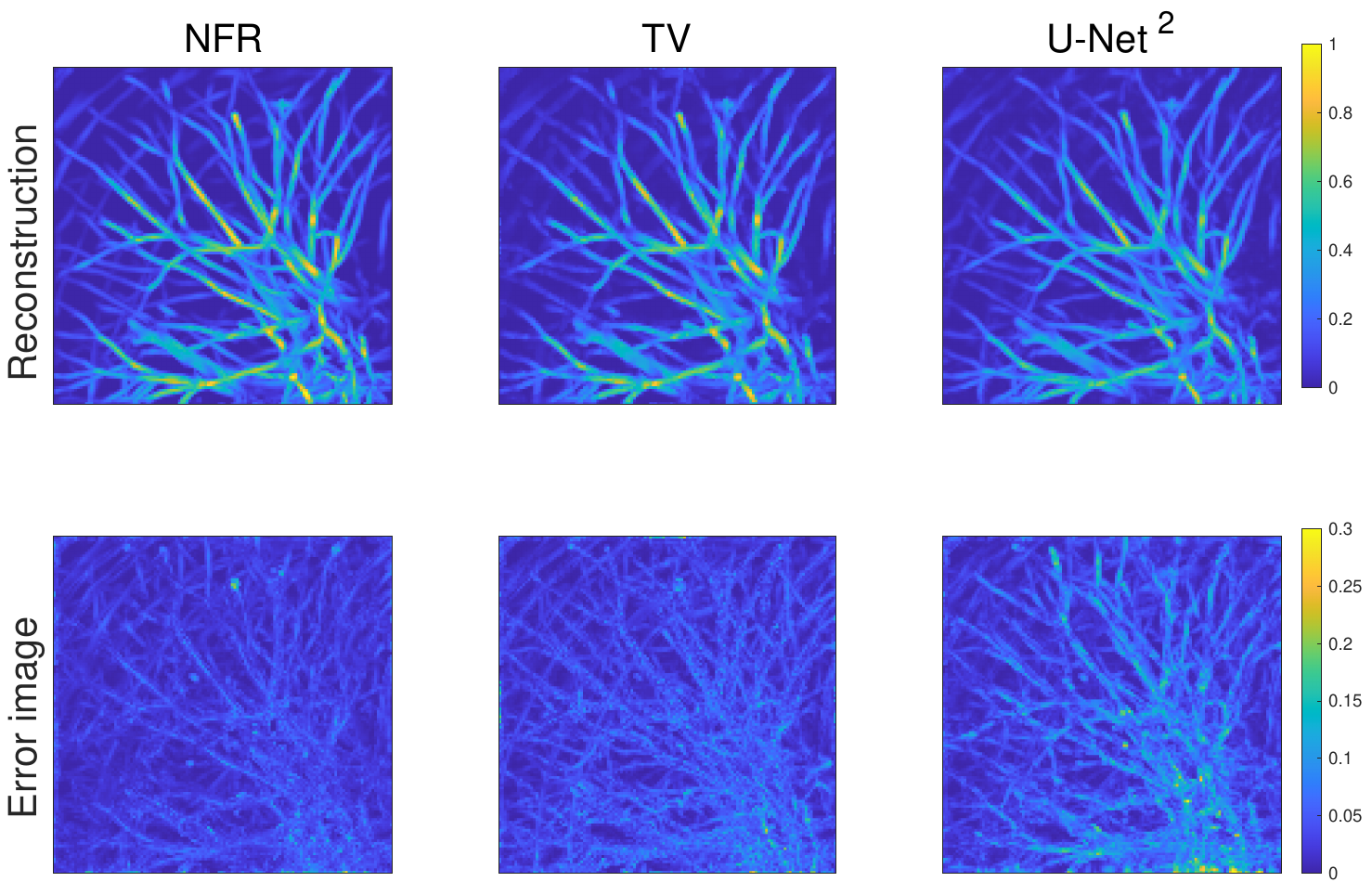}
  \caption{\small{\label{fig:4.5}Maximum projection intensity (MIP) along z-axis of the scaled reconstruction results $\vx_r$ and error images $\left|\vx_r-\vx^{true}\right|$ by \ac{NFR} method, \ac{TV} regularization and \Unet from the data with the number of azimuth and polar angles $(256, 16)$ and 0.05 Gaussian noise.}}
\end{figure}

\paragraph{Noise effect} 
Table \ref{tab:2} shows the quantitative comparison of reconstruction results among three methods from the measurement data with fixed sparsity $(64, 8)$ angles and varying noise levels 0.01 and 0.05. The reconstruction and the corresponding error images are shown in Figure \ref{fig:4.6}. \ac{NFR} method achieves the most high-quality reconstruction, followed by \ac{TV} in the second, while \Unet is the worst. As the noise level increases, the background exhibits stronger noise, leading to a degradation in the clarity of vascular structures with low values, resulting in visible blurring along their edges.

\begin{table}[!htb]
   \caption{\label{tab:2}Quantitative comparison of the reconstruction results by \ac{NFR} method, \ac{TV} method and \Unet post-processing from the data with 0.01 and 0.05 Gaussian noise and $(64, 8)$ azimuth and polar angles.}
  \centering
  \begin{tabu}{l|l||cc}
    \tabucline[2pt]{-}
     \multicolumn{2}{l}{}       & Noise 0.01&   Noise 0.05 \\
     \hline\hline
     \multirow{3}{*}{\ac{PSNR}}& \ac{NFR} &\bf{42.58}& \bf{41.38} \\
                          & \ac{TV}  & 40.72 &  39.36     \\
                          & \Unet$\!\!^1$ &  35.16&  34.78     \\
                          & \Unet$\!\!^2$ & 37.46  & 36.73       \\
                          & \Unet$\!\!^3$ & 33.58  & 32.21       \\
     \hline
      \multirow{3}{*}{\ac{SSIM}}& \ac{NFR} &\bf{0.97}& \bf{0.97} \\
                          & \ac{TV}  & 0.95  & 0.95     \\
                          & \Unet$\!\!^1$ & 0.93  &  0.89    \\
                          & \Unet$\!\!^2$ & 0.94  &  0.92      \\
                          & \Unet$\!\!^3$ & 0.91  &  0.88      \\
     \hline
     \multirow{3}{*}{\ac{RRA}}& \ac{NFR} &\bf{0.23}& \bf{0.27} \\
                          & \ac{TV}  & 0.29  &0.34     \\
                          & \Unet$\!\!^1$ &  0.55    & 0.57    \\     
                          & \Unet$\!\!^2$ & 0.42  &  0.46      \\
                          & \Unet$\!\!^3$ & 0.56  &  0.59      \\
    \tabucline[2pt]{-}
  \end{tabu}
\end{table}

\begin{figure}[!htb]
  \centering
  \includegraphics[trim=0 0 0 0,clip,width=\textwidth]{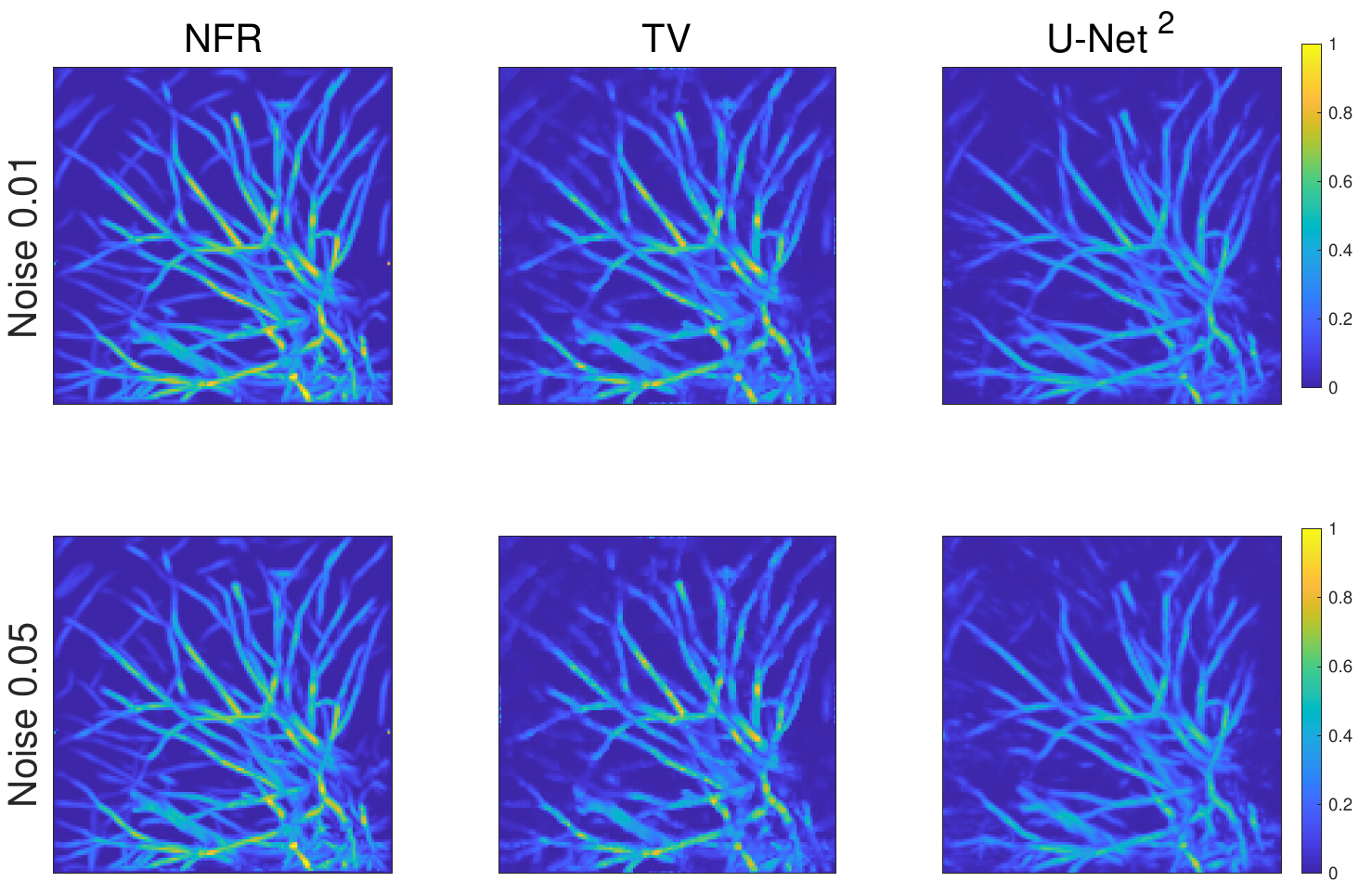}
  \caption{\small{\label{fig:4.6}Comparison of scaled reconstruction results among \ac{NFR} method, \ac{TV} regularization and \Unet from the data with the number of azimuth and polar angles $(64, 8)$, 0.01 and 0.05 Gaussian noise.}}
\end{figure}

\section{Conclusion}
\label{sec:6}

The proposed Normalizing Flow-Based Regularization (\ac{NFR}) achieves significant performance improvements in reconstructing three-dimensional photoacoustic tomography from limited-view measurements. The NFR method outperforms the classic Total Variation \ac{TV} method and \Unet post-processing. By utilizing data augmentation, the \Glow model accurately models the statistical properties of the training data, even with a limited amount of data. We propose an adaptive strategy to automatically tune the regularization parameter and prove that it is admissible. For deploying our method for images that are larger than the ones in the training data, a setting common in practical applications, considering the localized nature of vascular structures, we apply a patched-based regularization strategy. This involves randomly segmenting sub-patches of the same size as the training data, thereby avoiding the high computational costs of model training with high-resolution inputs and enhancing the model's flexibility as a regularizer. Compared to supervised network models that heavily depend on specific measurement settings, the NFR can be directly applied to any measurement scenario once trained.

To evaluate the robustness of \ac{NFR} against varying noise levels and sparsity in measurements, we compare it with \ac{TV} and \Unet in reconstructing larger images than those in the training data, under different noise and sparsity conditions. The results demonstrate that both \ac{NFR} and \ac{TV} consistently outperform \Unet across all scenarios. For relatively dense measurements with 256 azimuth angles and 16 polar angles, the \ac{TV} and \ac{NFR} methods achieve comparable accuracy. However, in the case of sparse measurements with 64 azimuth angles and 8 polar angles, the \ac{NFR} method significantly outshines its counterpart. Notably, even when the measurement configuration is reduced to $(64, 8)$ angles, \ac{NFR} maintains a similar reconstruction accuracy to the \ac{TV} method at $(128, 11)$ angles, thereby offering a computational saving of approximately one-third.

The effectiveness of the \ac{NFR} method primarily stems from its precise estimation of the image prior using the normalizing flow model \Glow. As a result, it is particularly suitable for image sets originating from the same distribution, such as medical images from identical biological tissues or imaging modalities. However, this also represents a limitation of the \ac{NFR} method. For more diverse image sets, a more complex normalizing model is necessary to accurately describe their distribution and adapt to the measurements. Despite this, the \ac{NFR} method holds significant potential to enhance the quality of image reconstruction while reducing the need for extensive measurements and computational resources.

\section*{Acknowledgments}
We would like to thank the anonymous referees for their useful comments that helped us improve the quality of the paper. This work was supported by the Ministry of Education under Grant MOE-000537-01.


\end{document}